\documentclass[twocolumn]{autart}
\usepackage{epsfig}
\usepackage{times}
\usepackage{amsmath}
\usepackage{amssymb}
\usepackage{enumerate}
\usepackage{cite}
\usepackage{color}
\usepackage{etoolbox}
\usepackage{changes}
\usepackage{xspace}
\let\classAND\AND
\let\AND\relax
\usepackage[ruled]{algorithm}
\usepackage{algorithmic}

\let\AND\classAND
\AtBeginEnvironment{algorithmic}{\let\AND\algoAND}

\newenvironment{proof}{\textbf{Proof:} }{\hfill$\square$}
\usepackage{graphicx}

\usepackage{subfigure}
\graphicspath{{./figures/}}
\usepackage[font=footnotesize,labelfont=bf]{caption}
\usepackage{epstopdf}
\usepackage{float}
\DeclareGraphicsExtensions{.png}
\usepackage{thmtools}

\newcommand{\mb}[1]{\mathbf{#1}}
\newcommand{\mbb}[1]{\mathbb{#1}}

\newcommand{\td}[1]{\tilde{#1}}
\newcommand{\mc}[1]{\mathcal{#1}}

\newtheorem{lemma}{Lemma}
\newtheorem{theorem}{Theorem}

\newtheorem{example}{Example}
\newtheorem{assumption}{Assumption}

\newtheorem{corollary}{Corollary}

\begin{document}

\begin{frontmatter}
\title{Distributed Safe Resource Allocation using Barrier Functions\thanksref{footnoteinfo}}

\thanks[footnoteinfo]{This work was supported in part by the funding from Digital Futures and in part by the Swedish Research Council (Vetenskapsr\r{a}det) under grant 2020-03607. Corresponding author: Xuyang Wu.}

\author[kth]{Xuyang Wu}\ead{xuyangw@kth.se},
\author[stockholm]{Sindri Magn\'{u}sson}\ead{sindri.magnusson@dsv.su.se},
\author[kth]{Mikael Johansson}\ead{mikaelj@kth.se}
\address[kth]{School of Electrical Engineering and Computer Science, KTH Royal Institute of Technology, Sweden}
\address[stockholm]{Department of Computer and System Science, Stockholm University, Sweden}


\begin{abstract}
Resource allocation plays a central role in many networked systems such as smart grids, communication networks and urban transportation systems. In these systems, many constraints have physical meaning and 
having feasible allocation 
is often vital to avoid system breakdown.
Hence, algorithms with asymptotic feasibility guarantees are often insufficient since it is impractical to run algorithms for an infinite number of rounds. This paper proposes a distributed feasible method (DFM) for safe resource allocation based on barrier functions. In DFM, every iterate is feasible and thus safe to implement. We prove that under mild conditions, DFM converges to an arbitrarily small neighbourhood of the optimal solution. Numerical experiments demonstrate the competitive performance of DFM.
\end{abstract}
\end{frontmatter}

\setlength{\abovedisplayskip}{3pt}
\setlength{\belowdisplayskip}{3pt}

\section{Introduction}

Resource allocation among cooperative agents in networked systems is a central problem in water distribution networks, urban transportation systems, and power networks. In these systems, an allocation is safe only when it is feasible, and violation of physical constraints may cause system breakdown. Consequently, it is important that resource allocation mechanisms always generate feasible allocations. 
%
%
\label{ssec:ME} To make these challenges more concrete, we consider three resource allocation problems from power networks and communications that currently lack distributed solutions.
 
\textbf{Economic dispatch \cite{Yang13}}: In a smart grid where some users can generate power and some users have demands, the users wish to cooperatively find the optimal power allocations:
\begin{equation}\label{eq:ED}
\begin{aligned}
& \underset{x_i\in \mathbb{R},i\in\mathcal{V}}{\text{minimize}}
& & \sum_{i\in\mathcal{V}} f_i(x) \\
& \text{subject to} 
& & \sum_{i\in\mathcal{V}} x_i = c,\\
&&& x_i\in \mathcal{X}_i,~~i \in \mathcal{V}
\end{aligned}
\end{equation}
where $x_i\in \mathbb{R}$ is the power injection of node $i\in\mathcal{V}$ (negative injection $x_i<0$ means that node $i$ consumes power from the grid). The cost function $f_i$ is the generation cost or the disutility related to shifting power demands for node $i$, and $c\in \mathbb{R}$ is the power injection not accounted for by the nodes in $\mathcal{V}$.  
The coupling constraint ensures that the supply meets the demand, which is a hard physical constraint in power systems. 
The local constraints $\mathcal{X}_i$ are important, as they specify hard constraints on devices or user preferences.

\textbf{Multiple resources \cite{Enyioha18}}: In smart grids, users might get power from different sources, e.g., renewable or coal, with different preference and the goal is to minimize user's total disutility. To model such scenarios, we need to consider multiple coupling constraints, e.g., 

\begin{equation}\label{eq:problemmultiple}
\begin{aligned}
& \underset{x_i\in \mathbb{R}^2, i\in\mathcal{V}}{\text{minimize}}
& & \sum_{i\in\mathcal{V}} f_i(x_i^{\text{renew}},x_i^{\text{coal}})\\
& \text{subject to}
& & \sum_{i\in\mathcal{V}} x_i^{\text{renew}} = 0,
~~\sum_{i\in\mathcal{V}} x_i^{\text{coal}} = 0,\\
&&& x_i\in \mathcal{X}_i,~i\in\mathcal{V}.
\end{aligned}
\end{equation}
Here, the local objective functions encode user preferences for different energy sources. 
This problem is similar to economic dispatch, but is more challenging to solve. In fact, as we will review in Section~\ref{ssec:literature}, most existing feasible methods can only handle a single coupling constraint.

\textbf{Rate control in networks \cite{kelly98}}: Consider a network with several transmitters where each user $i$ transmits along a single route. The goal is to find the end-to-end transmission rates $x_i$ that maximize the total user utility 
\begin{equation}\label{eq:ratecontrol}
\begin{split}
    \underset{x_i\in\mathbb{R},~i\in\mathcal{V}}{\operatorname{maximize}} ~~&~ \sum_{i\in\mathcal{V}} U_i(x_i)\\
    \operatorname{subject~to} ~&~\sum_{i\in \mc{T}_\ell} x_i \le c_\ell,\quad\forall \ell\in\mathcal{L},\\
    &~ x_i\ge 0,~\forall i\in\mathcal{V}.
\end{split}
\end{equation}
Here, $\mathcal{V}$ is the set of transmitters (users), $x_i$ and $U_i$ are the transmission rate and the utility function of user $i\in\mathcal{V}$, respectively, $\mathcal{L}$ is the set of communication links, and $c_\ell$ is the capacity of link $\ell\in\mathcal{L}$. Moreover, $\mc{T}_\ell$ is the set of transmitters whose data flow is routed across link $\ell$.


\subsection{Literature Review}\label{ssec:literature}

In the past few decades, many distributed algorithms \cite{Aybat16,Nedic17,Chang16,wu2021distributed,Lee05,Tan22,XiaoL06b,Necoara13,Ghadimi13,Ho80,Magnusson16,Magnusson18} have been proposed to solve resource allocation problems. Most of these methods have only asymptotic feasibility guarantees, i.e., they ensure that the limit of iterates is feasible. This is insufficient because in any practical implementation, only a finite number of iterations can be executed. Moreover, unlike centralized environments, where an infeasible point can be projected onto the feasible region, such projections are rarely implementable in networked systems since the constraints are often global while only local communication between neighbors is allowed.

To solve resource allocation problems in a safe way, a number of distributed feasible methods~\cite{Tan22,XiaoL06b,Necoara13,Ghadimi13,Ho80,Magnusson16,Magnusson18} have been proposed. These methods guarantee feasibility of all iterates. Hence, whenever such an algorithm is stopped, the iterate is safe to implement. Among these methods, \cite{Tan22,XiaoL06b,Necoara13,Ghadimi13} guarantee feasibility of all iterates, but do not allow for local constraints. However, as we have shown by the introductory examples above, resource allocation problems in physical systems often have local constraints.

The most closely related works in the literature are \cite{Ho80,Magnusson16,Magnusson18}, which allow for local constraints and yield feasible iterates at every iteration. Among them, \cite{Ho80} addresses problems with one global linear equality constraint while \cite{Magnusson16,Magnusson18} consider problems with one global linear inequality constraint. However, these algorithms have serious limitations. They all assume one-dimensional local decision variables, \cite{Ho80} requires the local constraints to be the set of non-negative real numbers, and \cite{Magnusson16,Magnusson18} rely on star networks. Due to these limitations, none of the approaches in ~\cite{Ho80,Magnusson16,Magnusson18} can solve the three motivating examples in Section \ref{ssec:ME}.

\subsection{Contribution}

Motivated by the lack of distributed feasible methods for solving resource allocation problems with local constraints on general networks, we propose a class of distributed feasible methods (DFM) based on barrier functions. We also propose a reachability condition, which is necessary for a general class of distributed feasible algorithms including DFM and those in \cite{XiaoL06b,Necoara13,Ghadimi13,Ho80} to converge to the optimum from an arbitrary initial allocation. Compared to existing feasible methods \cite{Ho80,Magnusson16,Magnusson18} that allow for local constraints, DFM can handle a broader range of problems and networks. Specifically, DFM allows for
\begin{enumerate}
    \item multi-dimensional variables, multiple coupling constraints, and general undirected networks, while \cite{Ho80,Magnusson16,Magnusson18} only consider scalar variables and one coupling constraints, and \cite{Magnusson16,Magnusson18} rely on star networks.
    \item non-convex objective functions, while the objective functions are required to be convex in \cite{Ho80,Magnusson16} and strongly convex in \cite{Magnusson18}.
\end{enumerate}
With the above features, DFM can solve all the three motivating examples in Section \ref{ssec:ME}.

A preliminary conference version of this work can be found in \cite{Wu21}, which also transforms the problem using barrier functions. However, both the algorithm and the  convergence results  in \cite{Wu21} are different from this article. In particular, the algorithm in \cite{Wu21} is a random algorithm, it requires nodes to share their local cost function with neighbors which may cause privacy issues, and it is only proved to converge asymptotically in expectation on convex problems. In contrast, DFM is deterministic, has \emph{deterministic convergence} with \emph{non-asymptotic} rates, and can solve \emph{non-convex} problems. In addition, DFM needs no sharing of local cost functions between nodes and \emph{avoids the related privacy issues}.

\subsection*{Paper Organization and Notation}

The outline of this paper is as follows: Section \ref{sec:prob} formulates the problem, clarifies the challenges, and transforms the problem using barrier functions. Section \ref{sec:alg} develops DFM to solve the transformed problem and Section \ref{sec:convana} analyses its convergence properties. Subsequently, Section \ref{sec:numerical} evaluates the practical performance of DFM in numerical experiments. Finally, Section \ref{sec:conclusion} concludes the paper.

\noindent\emph{Notation.} 
For any set $\mathcal{S}\subseteq \mathbb{R}^n$, $|\mathcal{S}|$ represents its cardinality. We define $\|\cdot\|$ as the $\ell_2$ vector norm and $\otimes$ as the Kronecker product. In addition, $\mathbf{0}_m, \mathbf{1}_m$, and $I_m$ are the $m$-dimensional all-zero vector, the $m$-dimensional all-one vector, and the $m\times m$ identity matrix, respectively; we ignore their subscripts when the dimension is clear from context. 
For any $x\in\mathbb{R}^n$ and positive semidefinite matrix $W$, we define $\|x\|_W=\sqrt{x^TWx}$. A function $f:\mathbb{R}^d\rightarrow \mathbb{R}$, it is $L$-smooth if $f$ is differentiable on $\mathbb{R}^d$ and its gradient $\nabla f$ is Lipschitz continuous with Lipschitz constant $L$, i.e., $\|\nabla f(x)-\nabla f(y)\|\le L\|x-y\|$ $\forall x,y\in \mathbb{R}^d$;  it is $\sigma-$strongly convex for some $\sigma>0$ if $f(y)-f(x)\ge \langle g_x, y-x\rangle+\frac{\sigma}{2}\|x-y\|^2$ for all $x,y\in \mathbb{R}^d$ and $g_x\in\partial f(x)$. Restrictions on the information exchange among agents will be described by an undirected graph $\mc{G}=(\mc{V}, \mc{E})$ with node set $\mc{V}$ and edge set $\mc{E}$. For each node $i\in\mc{V}$, we define $\mathcal{N}_i = \{j:\{i,j\}\in\mathcal{E}\}$ as its neighbor set and $\bar{\mathcal{N}}_i = \mathcal{N}_i\cup\{i\}$.

\section{Problem, Challenges, and Solution}\label{sec:prob}

\subsection{Problem Formulation}
Consider a network of $n$ nodes described by an undirected graph $\mathcal{G}=(\mathcal{V}, \mathcal{E})$, where $\mathcal{V}=\{1,\ldots,n\}$ is the set of vertices (representing nodes) and $\mathcal{E}\subseteq \mc{V}\times\mc{V}$ is the set of edges (representing links). Each node $i\in\mathcal{V}$ observes a local constraint set $\mc{X}_i\subset\mathbb{R}^{d_i}$, a local cost function $f_i:\mbb{R}^{d_i}\rightarrow \mathbb{R}$, and a local constraint matrix $A_i\in\mathbb{R}^{m\times d_i}$. The nodes aim at cooperatively finding the optimal resource allocation:
\begin{equation}\label{eq:prob}
	\begin{split}
		\underset{x_i\in \mc{X}_i, i\in\mathcal{V}}{\operatorname{minimize}}~~&\sum_{i\in\mathcal{V}} f_i(x_i)\\
		\operatorname{subject~to}~&\sum_{i\in\mathcal{V}} A_ix_i = c.
	\end{split}
\end{equation}
Note that~\eqref{eq:prob} also allows for coupling inequality constraints in the form of $\sum_{i\in\mc{V}} A_ix_i\le c$. These are simply handled by introducing local variables $y_i$ and local constraints on the form $A_ix_i\le y_i$ for all nodes $i\in {\mathcal V}$, in combination with the global constraint 
$\sum_{i\in\mc{V}} y_i=c$.  
For example, the coupling constraint in \eqref{eq:ratecontrol} can be re-written as
\begin{align}
x_i &\le y_{i\ell},\quad\forall \ell\in\mathcal{L}, \forall i\in \mc{T}_\ell,\\
\sum_{i\in \mc{T}_\ell} y_{i\ell} &= c_\ell,\quad\;\forall \ell\in\mathcal{L}.\label{eq:couplingconstraintcommunication}
\end{align}

For the theoretical analysis, we impose the following assumptions on problem \eqref{eq:prob}.
\begin{assumption}\label{asm:prob}
	The following conditions hold:
	\begin{enumerate}[(a)]
		\item Each $f_i$ is $L_i$-smooth for some $L_i>0$.
		\item Each $\mc{X}_i=\{x\in\mathbb{R}^{d_i}: g_i^j(x)\le 0,~j=1,\ldots,q_i\}$ for some convex and $\beta$-smooth functions $g_i^j:\mathbb{R}^{d_i}\rightarrow \mathbb{R}$. Moreover, each $g_i^j$ is $\beta_1$-Lipschitz continuous on $\mathcal{X}_i$.
		\item There exists $\{\td{x}_i\}_{i\in\mathcal{V}}$ such that $\sum_{i\in\mathcal{V}} A_i\td{x}_i = c$ and $g_i^j(\td{x}_i)<0$ for all $i\in\mathcal{V}$ and $j=1,\ldots,q_i$.
		\item The optimal value $f^\star$ of problem \eqref{eq:prob} is bounded from below, i.e., $f^\star>-\infty$.
	\end{enumerate}
\end{assumption}
Assumptions \ref{asm:prob}(a) and (d) are standard in optimization, especially when each $f_i$ is non-convex. Assumptions \ref{asm:prob}(b)--(c) enable the use of barrier functions for the problem transformation in Section \ref{ssec:problemtrans}.

\subsection{Challenges}\label{ssec:challenge}


Solving \eqref{eq:prob} with non-asymptotic feasibility guarantee in a distributed way is challenging and primal-dual methods generally can only ensure asymptotic feasibility. Existing distributed primal feasible methods for solving \eqref{eq:prob} include the weighted gradient method in \cite{Ho80,XiaoL06b}, the multi-step weighted gradient method in \cite{Ghadimi13}, and the coordinate-wise weighted gradient method in \cite{Necoara13}. However, none of them can handle general $A_i$'s or general convex $\mc{X}_i$'s. Specifically, \cite{Ho80,XiaoL06b,Ghadimi13} assume each $A_i=I$, \cite{Necoara13} requires each $A_i$ to be a non-zero scalar, \cite{Ho80} needs each $\mc{X}_i$ to be the set of non-negative scalars, and \cite{XiaoL06b,Ghadimi13,Necoara13} do not allow for local constraint sets. In this subsection, we will show by simple examples that a straightforward extension of the methods in \cite{Ho80,XiaoL06b,Ghadimi13} to handling general $A_i$'s or local constraint sets may fail even if 
$\mc{G}$ is connected. The multi-step weighted gradient method \cite{Ghadimi13} is more complicated and is not analysed here for simplicity.

\textbf{Extension to handle general $A_i$'s}: When $A_i=1$ for all $i\in {\mathcal V}$, the methods in \cite{Ho80,XiaoL06b,Necoara13} can be described as follows. Start from a feasible solution $\{x_i^0\}_{i\in\mc{V}}$. At each iteration, for an edge set $\mc{E}^k\subseteq\mc{E}$, find $p_{ij}^k$ $\forall \{i,j\}\in \mc{E}^k$ by solving
\begin{equation}\label{eq:pupdateoldlit}
\begin{split}
\operatorname{minimize}~~&~ f_i^k(x_i^k+p_{ji})+f_j^k(x_j^k+p_{ij})\\
\operatorname{subject~to}~&~p_{ji}+p_{ij} = 0,
\end{split}
\end{equation}
where $f_i^k(x_i)=f_i(x_i^k)+\langle \nabla f_i(x_i^k), x_i-x_i^k\rangle+\frac{L_i}{2}\|x_i-x_i^k\|^2$ is a quadratic surrogate function for $f_i$ around iterate $x_i^k$. The decision variable $p_{ij}$ describes the optimal re-allocation of resources from node $i$ to node $j$ (by the optimality conditions for~\eqref{eq:pupdateoldlit}, this quantity is proportional to the difference in the gradients of the loss functions of the nodes at their current iterates). For $\{i,j\}\in \mc{E} \setminus \mc{E}^k$, $p_{ij}^k=\mb{0}$. Then, the resource allocation is updated as
\begin{equation}\label{eq:xupdatelit}
    x_i^{k+1} = x_i^k + \sum_{j\in\mc{N}_i} w_{ij}p_{ij}^k,~\forall i\in\mc{V}.
\end{equation}
where $w_{ij}=w_{ji}>0$ are scalar weights. In the weighted gradient descent method \cite{Ho80,XiaoL06b}, $\mc{E}^k=\mc{E}$ and in the coordinate-wise weighted gradient method \cite{Necoara13},  $\mc{E}^k$ is randomly sampled from $\mc{E}$. Note that by~\eqref{eq:xupdatelit}, 
%
\begin{equation}\label{eq:feasiblesucc}
    \sum_{i\in\mc{V}} x_i^{k+1}-\sum_{i\in\mc{V}}x_i^k=\!\sum_{\{i,j\}\in\mc{E}}\!\!\!w_{ij}(p_{ij}^k+p_{ji}^k) = \mathbf{0}.
\end{equation}
Hence, if $\{ x_i^0\}_{i\in {\mathcal V}}$ satisfies the coupling constraint, then so will $\{ x_i^k \}_{i\in {\mathcal V}}$ for all future iterations $k\geq 0$.

For general $A_i$'s, it is natural to find a feasible resource-reallocation by replacing \eqref{eq:pupdateoldlit} by
\begin{equation}\label{eq:pupdatelit}
\begin{split}
\operatorname{minimize}~~&~ f_i^k(x_i^k+p_{ji})+f_j^k(x_j^k+p_{ij})\\
\operatorname{subject~to}~&~A_ip_{ji}+A_jp_{ij} = \mathbf{0}.
\end{split}
\end{equation}
Using a similar argument as above, the resource updates $p_{ij}$ and $p_{ji}$ ensure that future iterates $\{x_i^k\}_{i\in\mc{V}}$ satisfy the global coupling constraint by \eqref{eq:feasiblesucc}. However, as the next example shows, this algorithm could fail to converge to the optimum.
\begin{example}\label{ex:smooth}
Consider the line graph with four nodes:
\begin{figure}[!htb]
	\centering
	\includegraphics[scale=0.7]{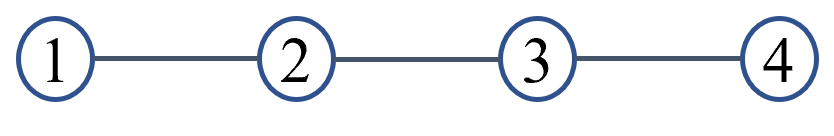}
\end{figure}
\vspace{-0.4cm}

The problem is
\begin{equation}\label{eq:counterexample}
	\begin{split}
		\underset{x_i\in\mbb{R},~i=1,\ldots,4}{\operatorname{minimize}}~&~ \sum_{i=1}^4 f_i(x_i)=\frac{1}{2}(x_i-\theta_i)^2\\
		\operatorname{subject~to}~&~x_1 + x_4 = 1,
	\end{split}
\end{equation}
where $(\theta_1,\theta_2,\theta_3,\theta_4) = (1,0,0,1)$. This is an example of \eqref{eq:prob} without local constraints. For this example, the point $(x_1', x_2', x_3', x_4')=(0,0,0,1)$ is a fixed point of the update \eqref{eq:xupdatelit} with $p_{ij}^k$, $\{i,j\}\in\mc{E}^k$ computing from \eqref{eq:pupdatelit}. To see this, note that by letting $x_i^k=x_i'$ $\forall i\in\mc{V}$, the solution to \eqref{eq:pupdatelit} is $p_{ij}^k=p_{ji}^k=\mb{0}$ for any $\{i,j\}\in\mc{E}$. However, the unique optimal solution of \eqref{eq:counterexample} is $(1/2, 0, 0, 1/2)$.
\end{example}
In Section \ref{sssec:reachable} below, we will analyse the issue in Example~\ref{ex:smooth}  and propose a reachability condition to handle it.

\textbf{Extension to handle general local constraint sets}: A straightforward extension of the above method to \eqref{eq:prob} with general local constraint sets is to add the constraints $x_i^k+p_{ji}\in \mc{X}_i$ and $x_j^k+p_{ij}\in \mc{X}_j$ to problem \eqref{eq:pupdatelit}, leading to
\begin{equation}\label{eq:pupdatelitextension}
\begin{split}
\operatorname{minimize}~~&~ f_i^k(x_i^k+p_{ji})+f_j^k(x_j^k+p_{ij})\\
\operatorname{subject~to}~&~A_ip_{ji}+A_jp_{ij} = \mathbf{0},\\
&~x_i^k+p_{ji}\in \mc{X}_i,~x_j^k+p_{ij}\in \mc{X}_j.
\end{split}
\end{equation}
The iterates $p_{ij}^k$ $\forall \{i,j\}\in \mc{E}^k$ are obtained by solving the above problem and the remaining parts are the same as earlier, i.e., $p_{ij}^k=p_{ji}^k=\mb{0}$ for $\{i,j\}\in\mc{E}\setminus \mc{E}^k$ and $\{x_i^{k+1}\}_{i\in\mc{V}}$ is updated by \eqref{eq:xupdatelit}. We additionally require $\sum_{j\in\mc{N}_i} w_{ij}\le 1$ for all $i\in\mc{V}$. Then, $x_i^{k+1}$ is a convex combination of $x_i^k$ and $x_i^k+p_{ji}^k$, $\forall j\in\mc{N}_i$, which all belong to $\mc{X}_i$. Hence, $x_i^{k+1}\in\mc{X}_i$ $\forall i\in\mc{V}$. By similar derivation in \eqref{eq:feasiblesucc}, $\{x_i^{k+1}\}_{i\in\mc{V}}$ is also feasible to the coupling constraint. Therefore, all iterates are feasible to \eqref{eq:prob}.

Below, we show that even if each $A_i=1$ which satisfies the conditions in \cite{Ho80,XiaoL06b,Necoara13}, the above method may not converge to the optimum.
\begin{example}\label{ex:counterconstraints}
Consider the same graph in Example \ref{ex:smooth}, the same objective function in \eqref{eq:counterexample}, and the constraints
\begin{equation*}
	\begin{split}
		\sum_{i=1}^4 x_i=1, x_i\in [0,1],~\forall i\in\{1,\ldots,4\}.
	\end{split}
\end{equation*}
For this example, the point $(x_1', x_2', x_3', x_4')=(0,0,0,1)$ is a fixed point of the extended method because by letting $x_i^k=x_i'$ $\forall i\in\mc{V}$, for any $\{i,j\}\in\mc{E}$, the solution to \eqref{eq:pupdatelitextension} is $p_{ij}^k=p_{ji}^k=\mb{0}$. However, the unique optimal solution is $(1/2, 0, 0, 1/2)$.
\end{example}
The failure in Example \ref{ex:counterconstraints} will be addressed in Section \ref{ssec:problemtrans} using barrier functions.

\subsection{Solution to the Challenges}

\subsubsection{Reachability of a general class of feasible methods}\label{sssec:reachable}

This subsection proposes a reachability condition. It enables a general class of distributed feasible methods to explore the whole feasible set, which is often necessary for guaranteeing optimality. The condition is also helpful for understanding and circumventing the failure in Example \ref{ex:smooth}. For convenience, let $\mb{x}^k=[(x_1^k)^T, \ldots, (x_n^k)^T]^T$ for all $k\ge 0$ and $A=[A_1, \ldots, A_n]$.

\textbf{General update scheme}: Start from a feasible iterate $\mb{x}^0$ of \eqref{eq:prob}. For any $k\ge 0$, update the iterate vector according to
\begin{equation}\label{eq:aggregate}
    \mb{x}^{k+1} = \mb{x}^k + \sum_{i\in\mc{V}} \mb{u}_i^k,
\end{equation}
where $\mb{u}_i^k=[(u_{i1}^k)^T, \ldots, (u_{in}^k)^T]^T$ and $u_{ij}^k\in\mbb{R}^{d_j}$ is the change that node $i$ makes on $x_j^k$, $j\in\mc{V}$. We assume that
\begin{equation*}
    \mb{u}_i^k\in\mc{S}_i = \{\mb{u}_i: u_{ij} = \mathbf{0}~\forall j\notin \bar{\mc{N}}_i,~\sum_{j\in\bar{\mc{N}}_i} A_ju_{ij} = \mathbf{0}\},
\end{equation*}
i.e., each node can only affect its own iterate and the iterates of its neighbors. When nodes only have information of themselves and their neighbors, this is a natural way to preserve feasibility of the coupling constraint. The extended method in Section \ref{ssec:challenge} for handling general $A_i$'s and the feasible methods in \cite{Ho80,XiaoL06b,Ghadimi13,Necoara13} all belong to this scheme.

Below we introduce the reachability condition. Note that by \eqref{eq:aggregate} and $\mb{u}_i^k\in\mc{S}_i$, $\mb{x}^{k+1}-\mb{x}^k\in \mc{S}_1+\ldots+\mc{S}_n$. Moreover, $\mc{S}_i\subseteq \operatorname{Null}(A)$ for all $i\in\mc{V}$. Hence, with the update \eqref{eq:aggregate}, \emph{the optimal solution $\mathbf{x}^\star$ is reachable from any $\mathbf{x}^0\in\operatorname{Null}(A)$ only if the following assumption holds}.
\begin{assumption}[Reachability]\label{asm:reachable}
It holds that
\begin{equation}\label{eq:identicalset}
    \mathcal{S}_1 + \ldots + \mathcal{S}_n = \operatorname{Null}(A).
\end{equation}
\end{assumption}
This assumption involves both the network $\mc{G}$ and the constraint matrix $A$. Under Assumption \ref{asm:reachable}, the update \eqref{eq:aggregate} can explore the whole set $\{\mb{x}: A\mb{x}=c\}$ while preserving feasibility of all iterates to the coupling constraint. However, if Assumption \ref{asm:reachable} fails to hold, the scheme is unable to reach $\mb{x}^{\star}$ from any initial allocation $\mb{x}^0$ that satisfies $\mb{x}^\star-\mb{x}^0\in \operatorname{Null}(A)\setminus (\mc{S}_1+\ldots+\mc{S}_n)$.

It can be verified that Example \ref{ex:smooth} fails to satisfy Assumption \ref{asm:reachable}. However, by simply adding the link $\{1,4\}$ to $\mc{G}$, Assumption \ref{asm:reachable} holds. Moreover, by the following Lemma, Assumption \ref{asm:reachable} is satisfied by the coupling constraint in problems \eqref{eq:ED}-\eqref{eq:problemmultiple} and the coupling constraint \eqref{eq:couplingconstraintcommunication}, which correspond to the three motivating examples in Section \ref{ssec:ME}.

\begin{lemma}\label{lemma:reachable}
    If all the $A_i$'s have full row rank and $\mc{G}$ is connected, then Assumption \ref{asm:reachable} holds. Moreover, if we define transmitters who share the same transmission link as neighbors in problem \eqref{eq:ratecontrol}, i.e., $\mathcal{E}=\{\{i,j\}:~i,j\in \mc{T}_\ell\text{ for some }\ell\in\mc{L}\}$, then \eqref{eq:couplingconstraintcommunication} satisfies Assumption \ref{asm:reachable}.
\end{lemma}
\begin{proof}
See Appendix \ref{ssec:prooflemmareachable}.
\end{proof}

In problems \eqref{eq:ED}-\eqref{eq:problemmultiple}, each $A_i=I$ which has full row rank. Then by Lemma \ref{lemma:reachable}, Assumption \ref{asm:reachable} holds when $\mc{G}$ is connected.



We also discuss Assumption \ref{asm:reachable} on the consensus constraint.
\begin{example}[consensus constraint]
If $\mc{G}$ is connected, then the constraint that every iterate should be in consensus, i.e. $x_1=x_2=\ldots=x_n$, 
can be rewritten as $\sum_{i=1}^n [\mc{L}_{\mc{G}}]_ix_i = \mathbf{0}$ where $[\mc{L}_{\mc{G}}]_i$ is the $i$th column of the graph Laplacian $\mc{L}_{\mc{G}}$ of $\mc{G}$. Hence, we can include this consensus constraint in problem \eqref{eq:prob} by letting each $A_i=[\mc{L}_{\mc{G}}]_i$. Assumption \ref{asm:reachable} holds if and only if $\mc{G}$ contains a star network, i.e., if $\bar{\mc{N}}_i=\mc{V}$ for some $i\in\mc{V}$. This follows since $\mc{S}_i=\{\mb{0}\}$ if $\bar{\mc{N}}_i\ne\mc{V}$ and $\mc{S}_i=\operatorname{Null}(A)$ otherwise. Therefore, when using the scheme \eqref{eq:aggregate} to ensure every iterates to be consensus, the algorithm can in general converge to the optimum of \eqref{eq:prob} from any feasible initial point only if $\mc{G}$ contains a star network.
\end{example}



\subsubsection{Problem Transformation using Barrier Function}\label{ssec:problemtrans}

As illustrated in Example \ref{ex:counterconstraints}, distributed feasible methods for solving \eqref{eq:prob} may get stuck in non-optimal points when local constraints are present. To solve this issue, we transform the local constraints in \eqref{eq:prob} using barrier functions:
\begin{equation}\label{eq:transprob}
	\begin{split}
		\underset{x_i\in \td{\mc{X}}_i, i\in\mathcal{V}}{\operatorname{minimize}}~~~&\sum_{i\in\mathcal{V}} f_i(x_i)+\rho B_i(x_i)\\
		\operatorname{subject~to}~&\sum_{i\in\mathcal{V}} A_ix_i = c,
	\end{split}
\end{equation}
where each $\td{\mc{X}}_i=\{x_i\in \mbb{R}^{d_i}:~g_i^j(x_i)<0,~\forall j=1,\ldots,q_i\}$ is the interior of $\mc{X}_i$ and each $B_i(x_i)=-\sum_{j=1}^{q_i}\frac{1}{g_i^j(x_i)}$ is the inverse barrier function of the constraints $g_i^j(x_i)\le 0$, $\forall j=1,\ldots,q_i$. Problem \eqref{eq:transprob} is an approximation of problem \eqref{eq:prob} whose approximation error will be analysed in Lemma \ref{lemma:optimalitygapconvex} in Section \ref{sec:convana}. In addition, every feasible solution of problem \eqref{eq:transprob} is also feasible to \eqref{eq:prob}.

For simplicity, we rewrite \eqref{eq:transprob} in the compact form:
\begin{equation}\label{eq:compactprob}
	\begin{split}
		\underset{\mb{x}\in \td{\mc{X}}}{\operatorname{minimize}}~~~& F(\mb{x}):=\sum_{i\in\mc{V}} f_i(x_i)+\rho B_i(x_i)\\
		\operatorname{subject~to}~&A\mb{x} = c,
	\end{split}
\end{equation}
where $\mb{x}=[x_1^T, \ldots, x_n^T]^T$, $\mc{\td{X}}=\td{\mc{X}}_1\times\ldots\times \td{\mc{X}}_n\subset\mathbb{R}^N$ with $N=\sum_{i\in\mathcal{V}} d_i$, $f(\mb{x})=\sum_{i\in\mathcal{V}} f_i(x_i)$, and $B(\mb{x})=\sum_{i\in\mathcal{V}}B_i(x_i)$.

\section{Distributed Feasible Method (DFM)}\label{sec:alg}

This section develops the DFM algorithm to solve problem \eqref{eq:prob}, based on the problem transformation described in Section~\ref{ssec:problemtrans}. In DFM, each node $i\in\mc{V}$ holds a local iterate $x_i^k\in\mbb{R}^{d_i}$, where $k\ge 0$ represents iteration index. For any $i\in\mc{V}$ and $k\ge 0$, we define
\begin{equation*}
	f_i^k(x_i)=f_i(x_i^k)+\langle\nabla f_i(x_i^k),x_i-x_i^k\rangle+\frac{L_i}{2}\|x_i-x_i^k\|^2
\end{equation*}
which is a convex approximation of $f_i$ around $x_i^k$.

The first step of DFM is initialization, in which all the nodes cooperatively set $x_i^0$ $\forall i\in\mc{V}$ such that $\mb{x}^0=[(x_1^0)^T, \ldots, (x_n^0)^T]^T$ is feasible to \eqref{eq:compactprob}. After the initialization, at each iteration $k\ge 0$, each node $i\in\mc{V}$ first solves $\{p_{ij}^k\}_{j\in\bar{\mc{N}}_i}$, which is the optimum of the following problem:
\begin{equation}\label{eq:pupdate}
	\begin{split}
		\underset{p_{ij}\in\mbb{R}^{d_j},~j\in \bar{\mc{N}}_i}{\operatorname{minimize}}~&~\sum_{j\in \bar{\mc{N}}_i}(f_j^k(x_j^k+p_{ij})+\rho B_j(x_j^k+p_{ij}))\\
		\operatorname{subject~to}~~&~\sum_{j\in \bar{\mc{N}}_i}A_jp_{ij}=\mb{0},\\
		&~x_j^k+p_{ij}\in \td{\mc{X}}_j,~\forall j\in\bar{\mc{N}}_i,
	\end{split}
\end{equation}
and then updates
\begin{equation}\label{eq:xupdate}
	x_i^{k+1} = x_i^k+\sum_{j\in\bar{\mc{N}}_i} \eta_j p_{ji}^k,
\end{equation}
where
\begin{equation}\label{eq:etavalue}
	\eta_j=\frac{1}{\max_{\ell\in\bar{\mc{N}}_j} |\bar{\mc{N}}_\ell|}.
\end{equation}

{
	\renewcommand{\baselinestretch}{1.05}
	\begin{algorithm} [tb]
		\caption{Distributed Feasible Method (DFM)}
		\label{alg:DRRA}
		\begin{algorithmic}[1]
			\small
			\STATE \textbf{Initialization:}
			\STATE All the nodes cooperatively choose a feasible solution $\mathbf{x}^0=[(x_1^0)^T, \ldots, (x_n^0)^T]^T$ of \eqref{eq:compactprob} and agree on $\rho>0$.
			\STATE Each node $i\in\mc{V}$ calculates $\eta_i$ by \eqref{eq:etavalue} through collecting $|\bar{\mc{N}}_j|$ from all its neighbors $j\in \mc{N}_i$.
			\FOR{$k=0,1,\ldots$}
			\FOR{each node $i\in\mathcal{V}$}
			\STATE broadcast $\nabla f_i(x_i^k)$ and $x_i^k$ to all neighbors $j\in\mathcal{N}_i$.
			\STATE upon receiving $\nabla f_j(x_j^k)$ and $x_j^k$ for all $j\in\mathcal{N}_i$, determine $p_{ij}^k$ $\forall j\in\bar{\mc{N}}_i$ by solving \eqref{eq:pupdate}.
			\STATE send each $\eta_ip_{ij}^k$ to $j\in\mc{N}_i$.
			\STATE upon receiving $\eta_jp_{ji}^k$ $\forall j\in\mc{N}_i$, compute $x_i^{k+1}$ by \eqref{eq:xupdate}.
			\ENDFOR
			\ENDFOR
		\end{algorithmic}
	\end{algorithm}
}

A detailed distributed implementation of DFM is given in Algorithm \ref{alg:DRRA}. In the initialization, each node $i\in\mc{V}$ shares $A_i$, $\mc{X}_i$, and $|\bar{\mc{N}}_i|$ with its neighbors $j\in\mc{N}_i$, which will be used to solve \eqref{eq:pupdate} and implement \eqref{eq:xupdate}. Sharing of $A_j$ or $|\mathcal{N}_j|$ between neighbors is common in algorithms \cite{XiaoL06b,Necoara13} for handling problem \eqref{eq:prob}. In addition, when all the $\mc{X}_j$'s and $A_j$'s are identical, this issue disappears.

At each iteration, the communication cost includes the broadcast of $\nabla f_i(x_i^k)$ and $x_i^k$ to every $j\in\mc{N}_i$ and the transmissions of $\eta_i p_{ij}^k$ between neighboring nodes $\{i,j\}\in\mc{E}$, and the main computational burden is to solve~\eqref{eq:pupdate}, which can be efficiently addressed by the interior point method \cite[Section 19]{Nocedal06}.

DFM can be cast into the general update scheme \eqref{eq:aggregate}. Moreover, it has a close connection to the weighted gradient method \cite{XiaoL06b} when the local constraint sets are absent.
\begin{rem}[Connection with weighted gradient method]\label{rem:wgm}
	Suppose that $\mc{X}_i=\mbb{R}$ and $A_i=1$ for all $i\in\mc{V}$. Then, by letting $L_i=L$ $\forall i\in\mathcal{V}$, DFM \eqref{eq:pupdate}--\eqref{eq:xupdate} becomes the weighted gradient descent method:
	\begin{equation}\label{eq:unconstrainedDFM}
		\mb{x}^{k+1} = \mb{x}^k-\frac{1}{L}H\nabla f(\mb{x}^k),
	\end{equation}
	where $H=\sum_{i\in\mc{V}} \eta_iH^i$ with $H^i=(h_{j\ell}^i)_{n\times n}$ defined as
	\begin{equation*}
		h_{j\ell}^i = \begin{cases}
			1-\frac{1}{|\bar{\mc{N}}_i|}, & \text{if } j=\ell,\\
			-\frac{1}{|\bar{\mc{N}}_i|}, & \text{if } j,\ell \in\bar{\mc{N}}_i \text{ and } j\ne \ell,\\
			0, & \text{otherwise}.
		\end{cases}
	\end{equation*}
	The update \eqref{eq:unconstrainedDFM} takes the same form compared to the typical weighted gradient method \cite{XiaoL06b} which solves the same problem, but has a different weight matrix $H$.
\end{rem}

\section{Convergence Analysis}\label{sec:convana}

This section derives convergence results for DFM. To this end, we first introduce the KKT condition for problem \eqref{eq:compactprob}, which is necessary for local optimality. A point $\mb{x}'\in\mbb{R}^N$ satisfies the KKT condition of problem \eqref{eq:compactprob} if
\begin{enumerate}
	\item $\mb{x}'$ is feasible to problem \eqref{eq:compactprob}.
	\item There exists $v\in\mbb{R}^m$ such that
	\begin{equation}\label{eq:kktcond}
		-\nabla F(\mb{x}')-A^Tv \in \mc{N}_{\td{\mc{X}}}(\mb{x}'),
	\end{equation}
	where $\mc{N}_{\td{\mc{X}}}(\mb{x}')$ is the normal cone of $\td{\mc{X}}$ at $\mb{x}'$.
\end{enumerate}
Because $\tilde{\mc{X}}$ is open, $\mathcal{N}_{\tilde{\mc{X}}}(\mb{x}')=\{\mathbf{0}\}$ for all $\mb{x}'\in \tilde{\mc{X}}$. Thus, the existence of $v\in\mbb{R}^m$ satisfying \eqref{eq:kktcond} is equivalent to
\begin{equation*}
\nabla F(\mb{x}') \in \operatorname{Range}(A^T).
\end{equation*}
For every $i\in {\mathcal V}$, let $P_i\in\mbb{R}^{N\times N}$ be the projection matrix of $\mc{S}_i$ and define $W=\sum_{i\in\mc{V}} \eta_iP_i$. Under Assumption \ref{asm:reachable}, $\operatorname{Null}(W)=\operatorname{Range}(A^T)$ and $\mathbf{x}'$ is a KKT point of \eqref{eq:compactprob} iff $\mathbf{x}'$ is feasible and $W\nabla F(\mathbf{x}')=\mathbf{0}$. Because $W$ is positive semi-definite, $W\nabla F(\mathbf{x}')=\mathbf{0}$ is equivalent to $\|\nabla F(\mathbf{x}')\|_W=0$.

Let $\mb{x}^k = [(x_1^k)^T, \ldots, (x_n^k)^T]^T$ for all $k\ge 0$, where $x_i^k$ $\forall i\in\mc{V}$ are generated by Algorithm \ref{alg:DRRA}.
\begin{theorem}\label{thm:convergence}
	Suppose that Assumptions \ref{asm:prob}--\ref{asm:reachable} hold. Then, $\mb{x}^k$ $\forall k\ge 0$ are feasible to problems \eqref{eq:compactprob} and \eqref{eq:prob}, and
	\begin{equation}\label{eq:theononconvex}
			\sum_{k=0}^\infty \|\nabla F(\mb{x}^k)\|_W^2\le 2(\rho L_B+L)(F(\mb{x}^0)-f^\star),
	\end{equation}
	where $L=\max_{i\in\mc{V}} L_i$ and $L_B=(4\beta_1^2((F(\mathbf{x}^0)-f^\star)/\rho)^3+2\beta((F(\mathbf{x}^0)-f^\star)/\rho)^2)\max_{i\in\mathcal{V}}q_i$. In addition,
	\begin{enumerate}
		\item if each $f_i$ is convex and the set $\mathcal{C}:=\{\mathbf{x}: \mathbf{x} \text{ is feasible to } \eqref{eq:compactprob} \text{ and } F(\mathbf{x})\le F(\mathbf{x}^0)\}$ is compact,
		\begin{equation}\label{eq:theoconvex}
			F(\mb{x}^k)-F^\star\le \frac{2(L\!+\!\rho L_B)R^2}{k\lambda_W},
		\end{equation}
		where $F^\star$ is the optimal value of problem \eqref{eq:compactprob}, $R$ is the diameter of $\mathcal{C}$, and $\lambda_W$ is the smallest nonzero eigenvalue of $W$.
		\item if each $f_i$ is $\sigma$-strongly convex for some $\sigma>0$,
		\begin{equation*}
			\begin{split}
	       	&F(\mb{x}^k)-F^\star\le\left(1\!-\!\frac{\sigma\lambda_W}{L\!+\!\rho L_B}\right)^k(F(\mb{x}^0)-F^\star).
			\end{split}
		\end{equation*}
	\end{enumerate}
\end{theorem}
\begin{proof}
	See Appendix \ref{ssec:convergenceproof}.
\end{proof}

In Theorem \ref{thm:convergence}, \eqref{eq:theononconvex} implies $\lim_{k\rightarrow +\infty} \|\nabla F(\mb{x}^k)\|_W^2 = 0$. In addition, according to \cite[Lemma 1]{Davis16}, \eqref{eq:theononconvex} also yields $\min_{0\le t\le k} \|\nabla F(\mb{x}^t)\|_W^2 = o(1/k)$. 

Below, we investigate the optimality gap between problems \eqref{eq:prob} and \eqref{eq:compactprob}, provided that all $f_i$'s are convex. Specifically, we study the range of $\rho$ for guaranteeing certain accuracy.
\begin{lemma}\label{lemma:optimalitygapconvex}
	Suppose that Assumptions \ref{asm:prob}--\ref{asm:reachable} hold and $f_i$ $\forall i\in\mathcal{V}$ are convex. Let $\tilde{\mathbf{x}}^\star$ be an optimal solution of \eqref{eq:compactprob} and $\underline{f}\le f^\star$.
	For any $\epsilon>0$, $f(\tilde{\mathbf{x}}^\star)-f^\star\le \epsilon$ if
		\begin{equation}\label{eq:rhorangeinv}
	    0<\rho\le\begin{cases}
	        \frac{\epsilon}{2B(\mb{x}')}, & f(\mathbf{x}')-\underline{f}\le \epsilon/2\\
	        \frac{\epsilon^2}{4(f(\mathbf{x}')-\underline{f})B(\mb{x}')}, & \text{otherwise}
	    \end{cases}
	\end{equation}
	for some feasible solution $\mathbf{x}'$ of \eqref{eq:compactprob}.
\end{lemma}
\begin{proof}
	See Appendix \ref{ssec:proofofoptimalitygap}.
\end{proof}

By Lemma \ref{lemma:optimalitygapconvex}, DFM can converge to arbitrarily small error bound of the optimal value. If some $\underline{f}\le f^\star$ is known, \eqref{eq:rhorangeinv} can be determined in a distributed way. For example, set $\mathbf{x}'=\mathbf{x}^0$ and let each node $i$ calculate $f_i(x_i^0)$, $B_i(x_i^0)$. Then, by adding all the $f_i(x_i^0)$ and $B_i(x_i^0)$ by distributed sum protocols (e.g., \cite{ChenJY06}), the quantities $f(\mathbf{x}^0)$ and $B(\mathbf{x}^0)$ are obtained by all the nodes, so the range of $\rho$ in \eqref{eq:rhorangeinv} is known.

\begin{corollary}\label{prop:conv}
    Suppose that Assumptions \ref{asm:prob}--\ref{asm:reachable} hold. For any $0<\epsilon<2(f(\mathbf{x}^0)-f^\star)$, set $\rho=\frac{\epsilon^2}{4(f(\mathbf{x}^0)-\underline{f})B(\mb{x}^0)}$ and let $k_\epsilon$ be the smallest $k\ge 0$ satisfying $f(\mathbf{x}^k)-f^\star\le 2\epsilon$. Then,
    \begin{enumerate}
    \item Convex: if each $f_i$ is convex and $\mathcal{C}$ defined in Theorem \ref{thm:convergence} is compact, then $k_\epsilon=O(\frac{1}{\epsilon^5\lambda_W})$.
	\item Strongly convex: if each $f_i$ is $\sigma$-strongly convex for some $\sigma>0$, then $k_\epsilon=O(\frac{1}{\sigma\lambda_W\epsilon^4}\ln\frac{1}{\epsilon})$.
	\end{enumerate}
\end{corollary}
\begin{proof}
Note that $\rho=O(\epsilon^2)$ and $\rho L_B=O(\epsilon^{-4})$, where $L_B=O(1/\rho^3)$ is defined in Theorem \ref{thm:convergence}. Then, combining Theorem \ref{thm:convergence} and Lemma \ref{lemma:optimalitygapconvex} yields the result.
\end{proof}

In Corollary \ref{prop:conv}, we only show the effect of important factors. The Lipschitz constant $L$ is usually small compared to $\rho L_B=O(\epsilon^{-4})$ and is therefore ignored. The slow convergence rates in Corollary \ref{prop:conv} are caused by the use of barrier functions for transforming the difficult problem \eqref{eq:prob}. However, the practical performance of DFM is much better than this worst-case bound, as will be shown in Section \ref{sec:numerical}.

\section{Numerical Experiment}\label{sec:numerical}

To verify the theoretical results in Section \ref{sec:convana} and demonstrate the practical performance of DFM, we evaluate DFM against alternative methods on the motivating examples in Section \ref{ssec:ME}. The experimental settings are detailed below.

\textbf{Experiment I (Economic Dispatch):} We consider problem~\eqref{eq:ED} and draw both problem data and communication network from the IEEE-118 bus system \cite{Zimmerman11}, where $f_i$ is a quadratic function, $\mathcal{X}_i:=[\ell_i, u_i]\subset\mathbb{R}$ is a closed interval, and the number of power generators is $n=54$. In addition to generators, the IEEE-118 bus system also includes several non-generators. The generators and non-generators form an undirected, connected communication graph $\mathcal{G}'$. To form the communication network ${\mathcal G}$ for the generators, we define two generators $i,j$ to be neighbors in $\mathcal{G}$ if there exists a path in $\mathcal{G}'$ between $i,j$ which does not include any other generators.

\textbf{Experiment II (Multiple Resources):} 
We consider problem~\eqref{eq:problemmultiple} and draw both problem data and communication network from the IEEE-118 bus system. We let each power generator be a renewable or a coal power generator with equal probability. We treat both generators and non-generators as users and set $n=118$. The user disutility function is $f_i(x_i^{\text{renew}}, x_i^{\text{coal}})=\alpha_i(x_i^{\text{renew}}+x_i^{\text{coal}}-D_i)^2+\beta_i(x_i^{\text{coal}})^2$, where $x_i^{\text{renew}}$ and $x_i^{\text{coal}}$ are renewable and coal power consumption of node $i$, respectively, $D_i$ is the power demand of node $i$, $\alpha_i>0$ is a parameter indicating the discomfort of user $i$ for changing its demand, and $\beta_i>0$ is a parameter capturing the disapproval of user $i$ of using non-renewable energy. We set $\mathcal{X}_i=\{x_i=(x_i^{\text{renew}}, x_i^{\text{coal}})^T: x_i\ge\tilde{\ell}_i\}$, where $\tilde{\ell}_i=(-u_i,0)^T$ if $i$ is a renewable power generator and $\tilde{\ell}_i=(0,-u_i)^T$ if $i$ is a coal power generator, where $u_i$ is the same as in Experiment I, and $\tilde{\ell}_i=(0,0)^T$ if $i$ is a non-generator. The communication graph is $\mathcal{G}'$ in Experiment I.

\textbf{Experiment III (Rate Control):} We consider problem~\eqref{eq:ratecontrol} where the network of routers and the capacity of each link are shown in Fig. \ref{fig:routing}. Each source $S_i$ sends data to $D_i$ via some links and routers. We treat source nodes whose routing paths include common links, i.e.  $\{S_1,S_2\}$, $\{S_2,S_3\}$, and $\{S_3,S_4\}$, as neighbors in the communication graph. We set each $U_i$ in problem \eqref{eq:ratecontrol} as a sigmoidal utility function $U_i(x_i) = \frac{p_i}{1+e^{-a_i(x_i-b_i)}}+q_i$ for some randomly generated $a_i,b_i,p_i>0$ and choose $q_i\in\mathbb{R}$ be such that $U_i(0)=0$. Although these utility functions are not concave, they have Lipschitz-continuous gradients and satisfy Assumption \ref{asm:prob}.
\begin{figure}
    \centering
    \includegraphics[scale=0.6]{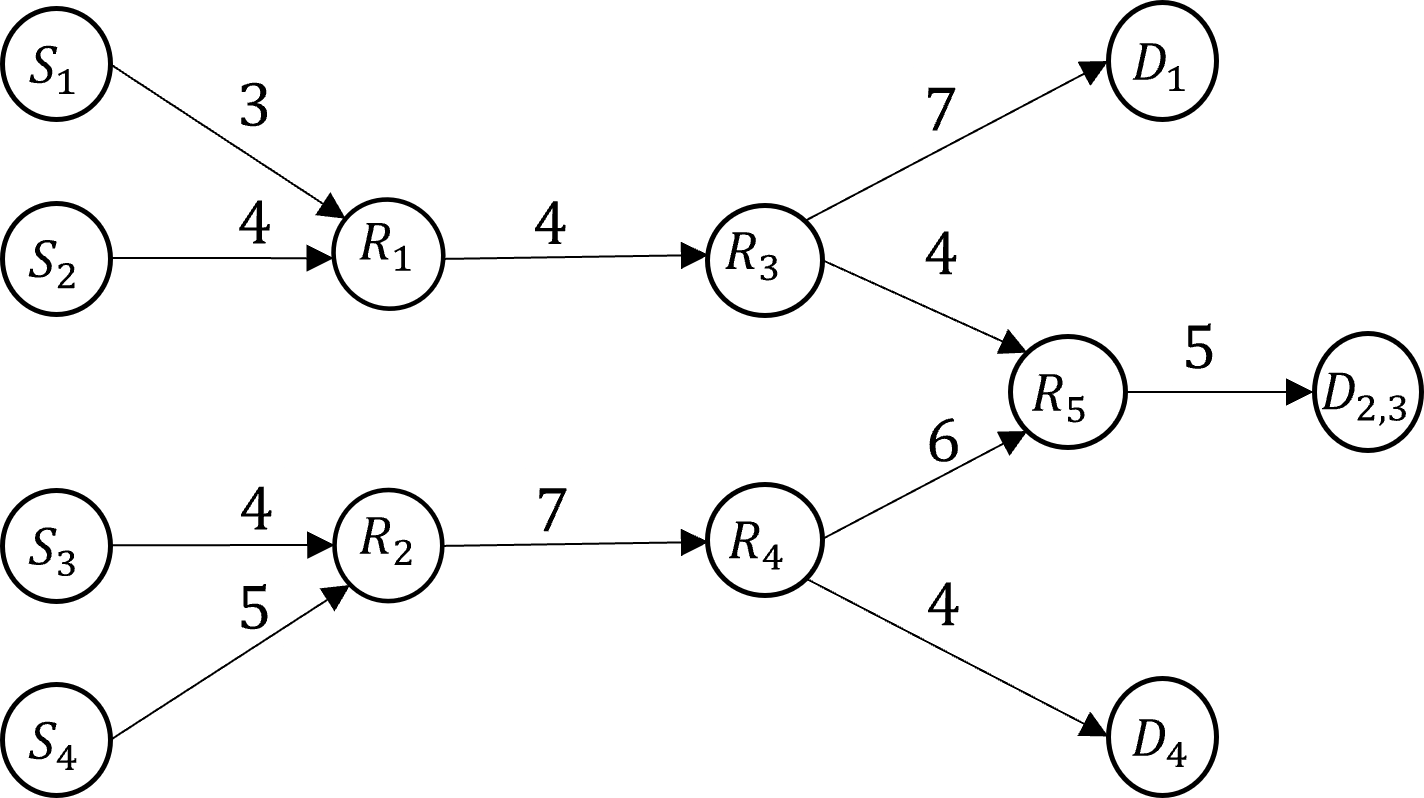}
    \caption{Routing network in Experiment III}
    \label{fig:routing}
\end{figure}

\begin{figure*}[!htb]
    \centering
    \subfigure[DFM]{\includegraphics[width=0.335\textwidth]{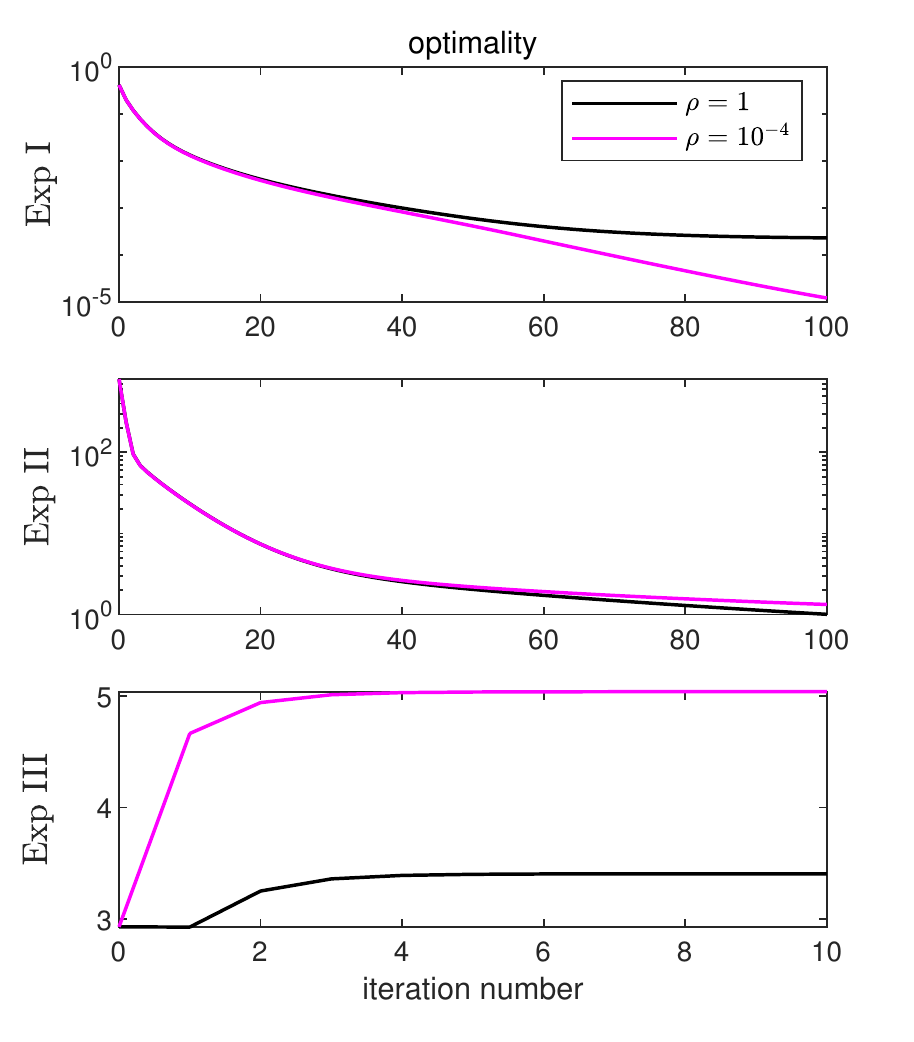}}
    \subfigure[Non-feasible methods]{\includegraphics[width=0.6\textwidth]{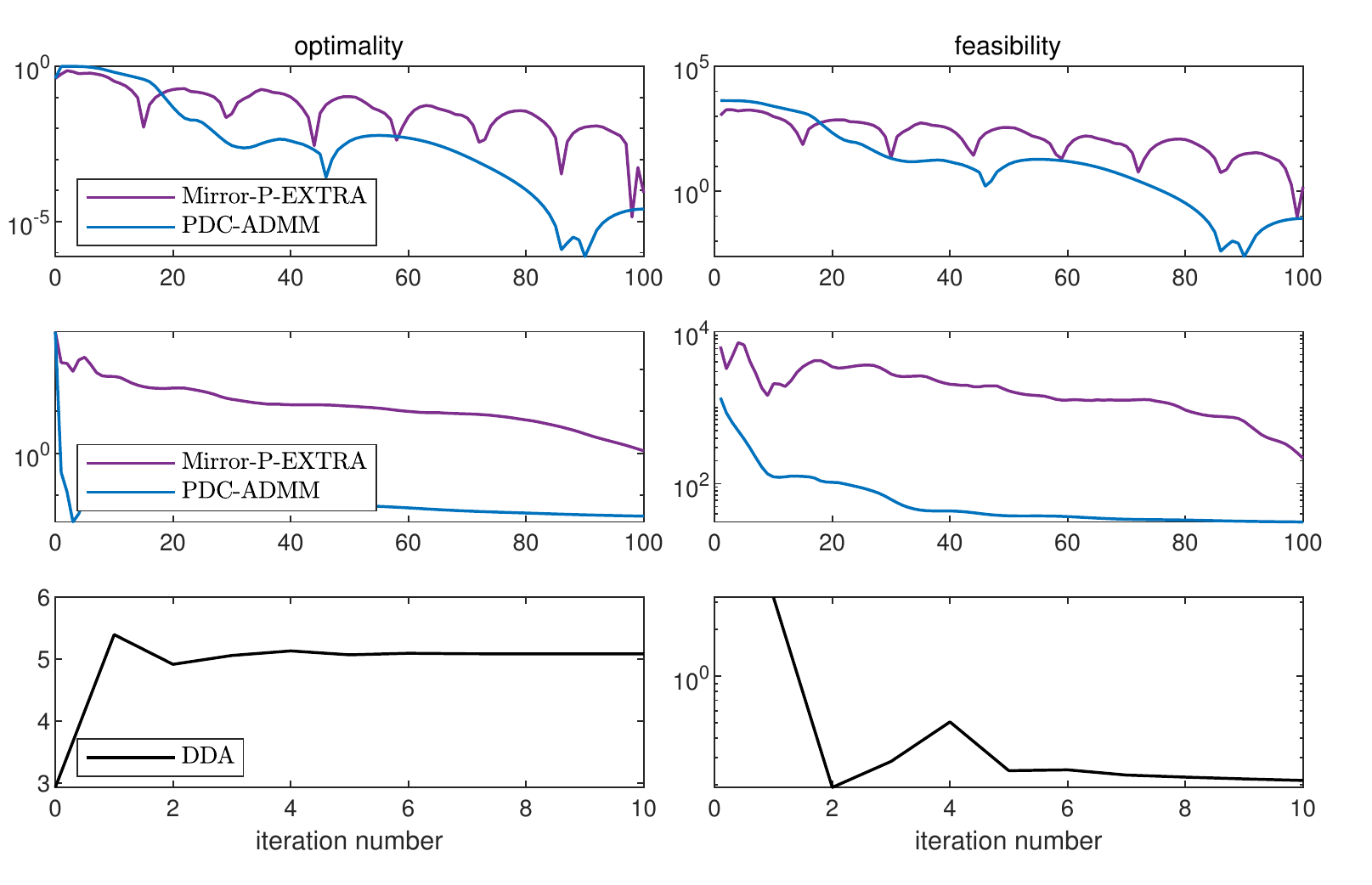}}
    \caption{Convergence of methods in Experiment I-III (Optimality represents objective error in Experiment I-II and objective value in Experiment III where the optimal value is unknown due to non-concavity of \eqref{eq:ratecontrol}. Feasibility means feasibility error in all experiments).}
    \label{fig:foobar}
\end{figure*}

As detailed in Section \ref{ssec:literature}, we have not found distributed feasible methods that can solve any of the three experiments. Hence, we consider non-feasible distributed methods Mirror-P-EXTRA \cite{Nedic17} and PDC-ADMM \cite{Chang16} in Experiment I-II and DDA \cite{Lee05} in Experiment III. These methods cannot guarantee feasibility of iterates.
For fair comparison, we choose the same initial primal iterate for all methods and tune all parameters within their theoretically justified ranges.

We display the convergence in all the three experiments in Fig. \ref{fig:foobar}, including DFM in Fig. \ref{fig:foobar}(a) and alternative non-feasible methods in Fig. \ref{fig:foobar}(b). The feasibility error of DFM is always $0$ and is thus ignored. Since the initial iterate is feasible, we plot the feasibility error starting from iteration $1$. In Fig. \ref{fig:foobar}, optimality represents the relative objective error $|(f(\mathbf{x}^k)-f^\star)/f^\star|$ in Experiment I-II and objective value of the maximization problem \eqref{eq:ratecontrol} in Experiment III. Since the sigmoidal function is non-concave, the optimal value of problem \eqref{eq:ratecontrol} is unknown. Moreover, since the local constraints are satisfied by all the simulated methods, we set $|\sum_{i=1}^n x_i^k-
c|$, $\|\sum_{i=1}^n x_i^k\|$, and $\sqrt{\sum_{\ell\in\mathcal{E}}(\max(0,\sum_{i\in \mc{T}_\ell} x_i-c_\ell))^2}$ as the feasibility error in Experiments I-III, respectively.

From Fig. \ref{fig:foobar}, we observe that the objective value of DFM converges quickly in all the experiments even compared with non-feasible methods, which demonstrates the competitive performance of DFM. Although in Experiment II, the objective value of PDC-ADMM converges faster than that of DFM, all non-feasible methods require a large number of iterations to reach small feasibility error. The experiment results also validate Theorem \ref{thm:convergence}, i.e., all iterates generated by DFM are feasible to \eqref{eq:prob} and converge under mild conditions. Moreover, from Figure \ref{fig:foobar}(a), we can see that smaller $\rho$ usually leads to higher accuracy.


\section{Conclusion}\label{sec:conclusion}
Our ability to allocate resources in a safe and efficient manner is critical to the operation of many physical network systems. We have proposed a distributed feasible method (DFM) to solve resource allocation problems while guaranteeing that all iterates are feasible and thus safe to implement. DFM is designed based on barrier functions and can handle more general problems (multiple coupling constraints, multi-dimensional variables, and non-convex objective functions) and network topologies than the previous state-of-the-art.

\section{Appendix}

\subsection{Proof of Lemma \ref{lemma:reachable}}\label{ssec:prooflemmareachable}
It's straightforward to see the left-hand side of \eqref{eq:identicalset} is a subset of $\operatorname{Null}(A)$ because $\mc{S}_i\subseteq \operatorname{Null}(A)$ for all $i\in\mathcal{V}$. Hence, to show \eqref{eq:identicalset}, we only need to prove
\begin{equation}\label{eq:subseteq}
    \operatorname{Null}(A)\subseteq \mc{S}_1 + \ldots + \mc{S}_n.
\end{equation}

First, we suppose that all the $A_i$'s have full row rank and prove \eqref{eq:subseteq} by showing that for any $\mathbf{p}\in\operatorname{Null}(A)$, there exist $\mb{q}_i\in\mc{S}_i$ $\forall i\in\mathcal{V}$ such that $\sum_{i\in\mc{V}} \mb{q}_i=\mb{p}$. Let $p_i\in\mbb{R}^{d_i}$ be the $i$th block of $\mb{p}$ and decompose each $p_i=u_i+v_i$, where $u_i\in\operatorname{Range}(A_i^T)$ and $v_i\in\operatorname{Null}(A_i)$. Let $y_i=A_iu_i$ for all $i\in\mathcal{V}$ and $\mb{y}=[(y_1)^T, \ldots, (y_n)^T]^T$. We have \begin{equation}\label{eq:sumyis0}
    \sum_{i\in\mc{V}} y_i=\sum_{i\in\mc{V}} A_iu_i=\sum_{i\in\mc{V}} A_ip_i=A\mb{p}=\mathbf{0}.
\end{equation}
Let $\mc{L}_{\mc{G}}$ be the graph Laplacian of $\mc{G}$. By \eqref{eq:sumyis0} and the connectivity of $\mc{G}$, we have $\mb{y}\in\operatorname{Range}(\mc{L}_{\mc{G}}\otimes I_m)$, i.e., there exists $\mathbf{z}$ satisfying
\begin{equation*}
    \mb{y}=(\mathcal{L}_{\mc{G}}\otimes I_m)\mb{z}.
\end{equation*}
Let $[\mb{q}_i]_j\in\mathbb{R}^{d_j},z_j\in\mbb{R}^m$ denote the $j$th block of $\mb{q}_i$ and $\mb{z}$, respectively. We construct the following $\mb{q}_i$:
\begin{equation*}
    [\mb{q}_i]_j = \begin{cases}
        A_i^\dag[\mc{L}_{\mc{G}}]_{ii}z_i+v_i, & i=j,\\
        A_j^\dag[\mc{L}_{\mc{G}}]_{ji}z_i, & \text{otherwise},
    \end{cases}
\end{equation*}
where $A_j^\dag = A_j^T(A_jA_j^T)^{-1}$ is the Moore-Penrose inverse of $A_j$. Because $[\mc{L}_{\mc{G}}]_{ji}=0$ for all $j\in\mc{V}\setminus \bar{\mc{N}}_i$, $A_jA_j^\dag=I$, and $\sum_{j\in\bar{\mc{N}}_i} [\mc{L}_{\mc{G}}]_{ji} = 0$, we have
\begin{equation}\label{eq:qiinsi}
    \mb{q}_i\in \mc{S}_i,~\forall i\in\mc{V}.
\end{equation}
To see $\sum_{i\in\mc{V}} \mb{q}_i=\mb{p}$, note that for each $j\in\mc{V}$,
\begin{equation*}
    \sum_{i\in\mc{V}} [\mb{q}_i]_j = v_j+A_j^\dag\sum_{i\in\bar{\mc{N}}_j} [\mc{L}_{\mc{G}}]_{ji}z_i = v_j+A_j^\dag y_j = v_j+u_j=p_j,
\end{equation*}
where $A_j^\dag y_j=u_j$ because $y_j=A_ju_j$ and $u_j\in\operatorname{Range}(A_j^T)$. Summarizing all the above gives \eqref{eq:subseteq}.

Next, we prove that \eqref{eq:couplingconstraintcommunication} satisfies \eqref{eq:identicalset}. Suppose $\mb{p}=\{p_{i\ell}\}_{\ell\in\mc{L}, i\in \mc{T}_\ell}$ is an arbitrary vector in $\operatorname{Null}(A)$. For any $\tilde{\ell}\in \mc{L}$, let $\mb{q}^{\tilde{\ell}}=\{q_{i\ell}^{\tilde{\ell}}\}_{\ell\in\mc{L}, i\in \mc{T}_\ell}$ be such that
\begin{equation*}
    q_{i\ell}^{\tilde{\ell}} = 
    \begin{cases}
        p_{i\ell}, & \ell = \tilde{\ell},\\
        0, & \text{otherwise}.
    \end{cases}
\end{equation*}
Then, $\mb{p}=\sum_{\tilde{\ell}\in\mc{L}} \mb{q}^{\tilde{\ell}}$. Moreover, by the definition of $\mc{S}_i$, $\mb{q}^{\tilde{\ell}}\in \mc{S}_i$ for all $i\in \mc{T}_{\tilde{\ell}}$. Hence, $\mb{p}\in \mc{S}_1+\ldots+\mc{S}_n$ and \eqref{eq:subseteq} holds. Completes the proof.

\subsection{Proof of Theorem \ref{thm:convergence}}\label{ssec:convergenceproof}
\subsubsection{Feasibility}
	We prove this by induction. Suppose that $\mb{x}^k$ is feasible to \eqref{eq:compactprob} for some $k\ge 0$, which holds naturally at $k=0$. Below, we show that $\mb{x}^{k+1}$ is feasible to \eqref{eq:compactprob}.
	

    By \eqref{eq:xupdate},
    \begin{equation}\label{eq:feasilocal}
        x_i^{k+1} = (1-\sum_{j\in\bar{\mathcal{N}}_i} \eta_j)x_i^k+\sum_{j\in\bar{\mathcal{N}}_i} \eta_j(x_i^k+p_{ji}^k),
    \end{equation}
    which, together with the feasibility of $\mathbf{x}^k$ to problem \eqref{eq:compactprob} and $\{p_{ij}^k\}_{j\in\bar{\mathcal{N}}_i}$ to problem \eqref{eq:pupdate}, ensures
    \begin{align}
        &x_i^k, x_i^k+p_{ji}^k \in\tilde{\mathcal{X}}_i,~\forall i\in\mathcal{V}, \forall j\in\bar{\mathcal{N}}_i,\label{eq:feasixkxkp}\\
    	&\sum_{i\in\mathcal{V}}A_ix_i^{k+1}=\sum_{i\in\mathcal{V}}A_i(x_i^k+\sum_{j\in\bar{\mc{N}}_i} \eta_j p_{ji}^k)\nonumber\\
    	=&\sum_{i\in\mathcal{V}}A_ix_i^k+\sum_{i\in\mathcal{V}}\eta_i\sum_{j\in\bar{\mathcal{N}}_i} A_jp_{ij}^k = c.\label{eq:feasieq}
    \end{align}
    Because $\eta_j=\frac{1}{\max_{\ell\in\bar{\mc{N}}_j} |\bar{\mc{N}}_\ell|}\le \frac{1}{|\bar{\mc{N}}_i|}$ for any $j\in\bar{\mathcal{N}}_i$, we have $1-\sum_{j\in\bar{\mathcal{N}}_i} \eta_j\ge 0$. Then from \eqref{eq:feasilocal}, \eqref{eq:feasixkxkp}, and the convexity of $\tilde{\mathcal{X}}_i$, we have $x_i^{k+1}\in\tilde{\mathcal{X}}_i$, which, together with \eqref{eq:feasieq}, yields the feasibility of $\mb{x}^{k+1}$ to problem \eqref{eq:compactprob}.

    Concluding all the above, all $\mb{x}^k$'s are feasible to \eqref{eq:compactprob}. Since all feasible solution of \eqref{eq:compactprob} is feasible to \eqref{eq:prob}, all $\mathbf{x}^k$'s are feasible to \eqref{eq:prob}.

\subsubsection{Non-convex case}

The proof consists of three parts. We first prove
\begin{equation}\label{eq:FksmallthanF0}
    F(\mathbf{x}^k)\le F(\mathbf{x}^0),~\forall k\ge 0.
\end{equation}
Next, we show that although each barrier function $B_i$ is not globally smooth, it is locally smooth. We derive the following using \eqref{eq:FksmallthanF0}: For all $i\in\mathcal{V}$ and $j\in\bar{\mc{N}}_i$,
\begin{equation}\label{eq:Bijresidualxik}
    \begin{split}
        &B_i(x_i^k)-B_i(x_i^k+p_{ji}^k)-\langle\nabla B_i(x_i^k+p_{ji}^k),-p_{ji}^k\rangle\\
        \ge &\frac{\|\nabla B_i(x_i^k)-\nabla B_i(x_i^k+p_{ji}^k)\|^2}{2L_B}.
    \end{split}
\end{equation}
Finally, we use \eqref{eq:Bijresidualxik} to show the result.

Define $F_i(x_i) = f_i(x_i)+\rho B_i(x_i)$ and $B_i^j(x_i) = \frac{1}{-g_i^j(x_i)}$.

\textbf{Part 1: proof of \eqref{eq:FksmallthanF0}.}
Define $\phi_i^k(x_i) = f_i^k(x_i)+\rho B_i(x_i)$ for each $i\in\mathcal{V}$ and $k\ge 0$. By the $L_i$-smoothness of each $f_i$, $f_i^k(x_i)\ge f_i(x_i)$ so that
\begin{align}
    & F_i(x_i)\le \phi_i^k(x_i),\quad \forall x_i\in\mathbb{R}^{d_i}.\label{eq:frhoBphi}
\end{align}
By \eqref{eq:frhoBphi}, the convexity of $\phi_i^k$, \eqref{eq:feasilocal}, and $\sum_{j\in\bar{\mathcal{N}}_i}\eta_j\le 1$,
\begin{equation}\label{eq:phikFk}
    \begin{split}
        &F(\mathbf{x}^{k+1})\le\sum_{i\in\mathcal{V}} \phi_i^k(x_i^{k+1})\\
        \le&\sum_{i\in\mathcal{V}}[(1-\!\!\sum_{j\in\bar{\mc{N}}_i} \eta_j)\phi_i^k(x_i^k)\!+\!\sum_{j\in\bar{\mc{N}}_i}\eta_j\phi_i^k(x_i^k\!+\!p_{ji}^k)].
    \end{split}
\end{equation}
Since $\{p_{ij}^k\}_{j\in\bar{\mathcal{N}}_i}$ is optimal to \eqref{eq:pupdate},
\begin{equation}\label{eq:sumphisumF}
    \sum_{j\in\bar{\mc{N}}_i} \phi_j^k(x_j^k+p_{ij}^k) \le \sum_{j\in\bar{\mc{N}}_i} \phi_j^k(x_j^k),
\end{equation}
and therefore
\begin{equation*}
\begin{split}
    &\sum_{i\in\mathcal{V}}\sum_{j\in\bar{\mc{N}}_i} \eta_j \phi_i^k(x_i^k+p_{ji}^k)=\sum_{i\in\mathcal{V}}\eta_i\sum_{j\in\bar{\mc{N}}_i}  \phi_j^k(x_j^k+p_{ij}^k)\\
    &\le\sum_{i\in\mathcal{V}}\eta_i\sum_{j\in\bar{\mc{N}}_i} \phi_j^k(x_j^k)=\sum_{i\in\mathcal{V}}\Big(\sum_{j\in\bar{\mc{N}}_i} \eta_j\Big) \phi_i^k(x_i^k).
\end{split}
\end{equation*}
Substituting the above equation into \eqref{eq:phikFk} and using $F_i(x_i^k)=\phi_i^k(x_i^k)$ $\forall i\in\mathcal{V}$ yields $F(\mathbf{x}^{k+1})\le \sum_{i\in\mathcal{V}} \phi_i^k(x_i^k)=F(\mathbf{x}^k)$ for all $k\ge 0$, which further implies \eqref{eq:FksmallthanF0}.

\textbf{Part 2: proof of \eqref{eq:Bijresidualxik}.} The proof starts from the local smoothness property of $B_i^j$.
\begin{lemma}\label{lemma:smooth}
    Suppose that $g_i^j$, $i\in\mc{V}$, $j\in\{1,\ldots,q_i\}$ satisfies Assumption \ref{asm:prob}(b). For any $\mathcal{M}>0$ and $x_i,y_i\in\tilde{\mathcal{X}}_i$, if $B_i^j(x_i)\le \mathcal{M}$ and $B_i^j(y_i)\le \mathcal{M}$, then
    \begin{equation}\label{eq:Bijresidual}
        \begin{split}
            &B_i^j(y_i) - B_i^j(x_i)-\langle \nabla B_i^j(x_i), y_i-x_i\rangle\\
            \ge &\frac{1}{8\beta_1^2\mathcal{M}^3+4\beta\mathcal{M}^2}\|\nabla B_i^j(y_i)-\nabla B_i^j(x_i)\|^2.
    \end{split}
\end{equation}
\end{lemma}
\begin{proof}
Define $u=\big|\frac{1}{(g_i^j(y_i))^2}-\frac{1}{(g_i^j(x_i))^2}\big|$, $v=\|\nabla g_i^j(y_i)-\nabla g_i^j(x_i)\|$, and $\bar{g} = \max(|g_i^j(x_i)|,|g_i^j(y_i)|)$. Also define $\Delta = B_i^j(y_i)-B_i^j(x_i)-\langle \nabla B_i^j(x_i), y_i-x_i\rangle$. To derive \eqref{eq:Bijresidual}, we first bound $\|\nabla B_i^j(x_i)-\nabla B_i^j(y_i)\|^2$ by $u,v$ and then bound $u,v$ by $\Delta$.

Note that for any $c_1>c_2>0$ and $z_1,z_2\in\mathbb{R}^d$,
\begin{equation}\label{eq:zc}
\begin{split}
        \Big\|\frac{z_1}{c_1}-\frac{z_2}{c_2}\Big\| &=\Big\|\frac{z_1-z_2}{c_1}+\left(\frac{1}{c_1}-\frac{1}{c_2}\right)z_2\Big\|\\
    &\le \frac{\|z_1-z_2\|}{c_1}+\|z_2\|\cdot\Big|\frac{1}{c_1}-\frac{1}{c_2}\Big|.
\end{split}
\end{equation}
Then, by letting
\begin{equation*}
    (z_1, c_1)=\begin{cases}
        (\nabla g_i^j(x_i), (g_i^j(x_i))^2) & (g_i^j(x_i))^2\ge (g_i^j(y_i))^2,\\
        (\nabla g_i^j(y_i), (g_i^j(y_i))^2), & \text{otherwise}
    \end{cases}
\end{equation*}
in \eqref{eq:zc} and using the definition of $B_i^j$,
\begin{align}
    &\|\nabla B_i^j(y_i)-\nabla B_i^j(x_i)\|^2=\bigg\|\frac{\nabla g_i^j(y_i)}{(g_i^j(y_i))^2}-\frac{\nabla g_i^j(x_i)}{(g_i^j(x_i))^2}\bigg\|^2\nonumber\\
    \le& (v/\bar{g}^2+\beta_1u)^2\le 2\beta_1^2u^2+2v^2/\bar{g}^4\label{eq:graBijexpansion},
\end{align}
where the second step uses the $\beta_1$-Lipschitz continuity of $g_i^j$.

By the definition of $B_i^j$,
\begin{equation}\label{eq:deltaxy}
	\begin{split}
		\Delta=&\frac{1}{g_i^j(x_i)}\!-\!\frac{1}{g_i^j(y_i)}\!-\!\frac{(\nabla g_i^j(x_i))^T(y_i\!-\!x_i)}{(g_i^j(x_i))^2}\\
		=&\frac{g_i^j(y_i)\!-\!g_i^j(x_i)}{g_i^j(x_i)g_i^j(y_i)}\!-\!\frac{(\nabla g_i^j(x_i))^T(y_i\!-\!x_i)}{(g_i^j(x_i))^2}.
	\end{split}
\end{equation}
Since $g_i^j(x_i)<0$ and $g_i^j(y_i)<0$, if $g_i^j(y_i)\ge g_i^j(x_i)$, then $|g_i^j(y_i)|\le |g_i^j(x_i)|$ and
\begin{equation*}
	\frac{g_i^j(y_i)-g_i^j(x_i)}{g_i^j(x_i)g_i^j(y_i)} \ge \frac{g_i^j(y_i)-g_i^j(x_i)}{(g_i^j(x_i))^2}.
\end{equation*}
Otherwise, we have $|g_i^j(x_i)|\le |g_i^j(y_i)|$ and, due to the convexity of $g_i^j$, $\nabla g_i^j(x_i)^T(y_i-x_i)\le g_i^j(y_i)-g_i^j(x_i)\le 0$ and therefore
\begin{equation*}
	-\frac{(\nabla g_i^j(x_i)^T(y_i-x_i)}{(g_i^j(x_i))^2}\ge -\frac{(\nabla g_i^j(x_i))^T(y_i-x_i)}{g_i^j(x_i)g_i^j(y_i)}.
\end{equation*}
Summarizing the above two cases and using \eqref{eq:deltaxy}, we have
\begin{align}
	\Delta&\ge\frac{g_i^j(y_i)-g_i^j(x_i)-(\nabla g_i^j(x_i))^T(y_i-x_i)}{-g_i^j(x_i)\max\{|g_i^j(x_i)|,|g_i^j(y_i)|\}}\nonumber\\
	&\ge \frac{\|\nabla g_i^j(x_i)-\nabla g_i^j(y_i)\|^2}{2\beta\bar{g}^2} = \frac{v^2}{2\beta\bar{g}^2}\label{eq:deltav},
\end{align}
where the second step is due to $g_i^j(y_i)-g_i^j(x_i)-\langle\nabla g_i^j(x_i), y_i-x_i\rangle \ge \frac{1}{2\beta}\|\nabla g_i^j(x_i)-\nabla g_i^j(y_i)\|^2$ from the smoothness of $g_i^j$ and \cite[equation (2.1.7)]{Nesterov04}.

Substituting $(\nabla g_i^j(x_i))^T(y_i-x_i)\le g_i^j(y_i)-g_i^j(x_i)$ into the first step of \eqref{eq:deltaxy} ensures
\begin{equation}\label{eq:deltau}
	\begin{split}
		\Delta\ge& \frac{g_i^j(y_i)\!-\!g_i^j(x_i)}{g_i^j(x_i)}\left(\frac{1}{g_i^j(y_i)}\!-\!\frac{1}{g_i^j(x_i)}\right)\\
		=&\frac{-(g_i^j(x_i)g_i^j(y_i))^2g_i^j(y_i)}{(g_i^j(x_i)+g_i^j(y_i))^2}u^2\\
		\ge& \frac{|g_i^j(y_i)|\min\{(g_i^j(y_i))^2, (g_i^j(x_i))^2\}}{4}u^2,
	\end{split}
\end{equation}
where the last step is due to $(g_i^j(x_i)+g_i^j(y_i))^2\le 4\max\{(g_i^j(x_i))^2,(g_i^j(y_i))^2\}$ and $g_i^j(y_i)<0$. In addition, because $B_i^j(x_i), B_i^j(y_i)\le \mathcal{M}$,
\begin{equation}\label{eq:gineq}
    \bar{g} \ge \min\{|g_i^j(y_i)|, |g_i^j(x_i)|\} \ge \frac{1}{\mathcal{M}}.
\end{equation}
By \eqref{eq:graBijexpansion}, \eqref{eq:deltav}, \eqref{eq:deltau}, and \eqref{eq:gineq}, we obtain
\begin{equation*}
    \begin{split}
        &\|\nabla B_i^j(y_i)-\nabla B_i^j(x_i)\|^2\\
        \le& \frac{8\beta_1^2\Delta}{|g_i^j(y_i)|\min\{(g_i^j(y_i))^2, (g_i^j(x_i))^2\}}+\frac{4\beta\Delta}{\bar{g}^2}\\
        \le&(8\beta_1^2\mathcal{M}^3+4\beta\mathcal{M}^2)\Delta,
    \end{split}
\end{equation*}
i.e., \eqref{eq:Bijresidual} holds.
\end{proof}

Since $B_i(x_i)=\sum_{j=1}^{q_i} B_i^j(x_i)$, we have from \eqref{eq:Bijresidual} that \begin{equation}\label{eq:Biresidual}
    \begin{split}
        &B_i(y_i) - B_i(x_i)-\langle \nabla B_i(x_i), y_i-x_i\rangle\\
        \ge &\sum_{j=1}^{q_i} \frac{1}{8\beta_1^2\mathcal{M}^3+4\beta\mathcal{M}^2}\|\nabla B_i^j(y_i)-\nabla B_i^j(x_i)\|^2\\
        \ge &\frac{1}{(8\beta_1^2\mathcal{M}^3+4\beta\mathcal{M}^2)q_i}\|\nabla B_i(y_i)-\nabla B_i(x_i)\|^2.
    \end{split}
\end{equation}

To derive \eqref{eq:Bijresidualxik} using Lemma \ref{lemma:smooth}, we show that for any $i\in\mathcal{V}$,
\begin{align}
    & B_i(x_i^k)\le(F(\mathbf{x}^0)-f^\star)/\rho,\label{eq:BijxM}\\
    & B_j(x_j^k+p_{ij}^k)\le(F(\mathbf{x}^0)-f^\star)/\rho,~\forall j\in\bar{\mc{N}}_i.\label{eq:BijxpM}
\end{align}
Since $\mathbf{x}^k$ is feasible to problem \eqref{eq:prob}, $f(\mathbf{x}^k)\ge f^\star$ and $B_j(x_j^k)>0$ $\forall j\in\mathcal{V}$ which indicates $B(\mathbf{x}^k)=\sum_{j\in\mathcal{V}} B_j(x_j^k)\ge B_i(x_i^k)$. This, together with \eqref{eq:FksmallthanF0}, implies
\begin{equation*}
    B_i(x_i^k)\le B(\mathbf{x}^k) = \frac{F(\mathbf{x}^k)-f(\mathbf{x}^k)}{\rho}\le \frac{F(\mathbf{x}^0)-f^\star}{\rho},
\end{equation*}
i.e., \eqref{eq:BijxM} holds. By \eqref{eq:frhoBphi}, \eqref{eq:sumphisumF}, and $\phi_j^k(x_j^k)=F_j(x_j^k)$,
\begin{equation*}
\begin{split}
    F(\mathbf{x}^k)&=\sum_{j\in\mathcal{V}\setminus\bar{\mathcal{N}}_i}F_j(x_j^k)+\sum_{j\in\bar{\mathcal{N}}_i} \phi_j^k(x_j^k)\\
    &\ge\sum_{j\in\mathcal{V}\setminus\bar{\mathcal{N}}_i}F_j(x_j^k)+\sum_{j\in\bar{\mathcal{N}}_i} \phi_j^k(x_j^k+p_{ij}^k)\\    &\ge\sum_{j\in\mathcal{V}\setminus\bar{\mathcal{N}}_i}F_j(x_j^k)+\sum_{j\in\bar{\mathcal{N}}_i} F_j(x_j^k+p_{ij}^k)\\
    &\ge f^\star+\rho\sum_{j\in\mathcal{V}\setminus\bar{\mathcal{N}}_i}B_j(x_j^k)+\rho\sum_{j\in\bar{\mathcal{N}}_i} B_j(x_j^k+p_{ij}^k),
\end{split}
\end{equation*}
where the last step uses the feasibility of $\{z_j\}_{j\in\mathcal{V}}$ where $z_j=x_j^k+p_{ij}^k$ if $j\in\bar{\mathcal{N}}_i$ and $z_j=x_j^k$ otherwise. Combining the above equation with \eqref{eq:FksmallthanF0} and using $B_j(x_j^k)>0$ $\forall j\in\mathcal{V}\setminus\bar{\mathcal{N}}_i$ and  $B_j(x_j^k+p_{ij}^k)>0$ $\forall j\in\bar{\mathcal{N}}_i$, we obtain \eqref{eq:BijxpM}.

By \eqref{eq:BijxM}--\eqref{eq:BijxpM}, for any $i\in\mathcal{V}$, $j\in\bar{\mathcal{N}}_i$, and $\ell\in\{1,\ldots,q_i\}$, $B_i^\ell(x_i^k)\le B_i(x_i^k)\le(F(\mathbf{x}^0)-f^\star)/\rho$ and $B_i^\ell(x_i^k+p_{ji}^k)\le (F(\mathbf{x}^0)-f^\star)/\rho$. Then, by \eqref{eq:Biresidual} with $\mc{M}=(F(\mb{x}^0)-f^\star)/\rho$, \eqref{eq:Bijresidualxik} holds.

\textbf{Part 3: descent of objective value.} The proof in this part includes two steps. Step 1 proves
\begin{equation}\label{eq:Fxk1minusFk}
	\begin{split}
		&F(\mb{x}^{k+1})-F(\mb{x}^k)\\
		\le\!&-\!\frac{\sum_{i\in\mc{V}}\eta_i\sum_{j\in\bar{\mc{N}}_i}\|\nabla \phi_j^k(x_j^k)\!-\!\nabla \phi_j^k(x_j^k+p_{ij}^k)\|^2}{2(L+\rho L_B)}.
	\end{split}
\end{equation}
and step 2 derives
\begin{equation}\label{eq:Wweight}
    \sum_{i\in\mc{V}}\eta_i\sum_{j\in\bar{\mc{N}}_i}\|\nabla \phi_j^k(x_j^k)\!-\!\nabla \phi_j^k(x_j^k+p_{ij}^k)\|^2\ge \|\nabla F(\mathbf{x}^k)\|_W^2.
\end{equation}
From \eqref{eq:Fxk1minusFk} and \eqref{eq:Wweight},
\begin{equation}\label{eq:convexsuccessivedescent}
	\begin{split}
		F(\mb{x}^{k+1})-F(\mb{x}^k)\le&-\frac{\|\nabla F(\mb{x}^k)\|_W^2}{2(L+\rho L_B)},
	\end{split}
\end{equation}
which further leads to \eqref{eq:theononconvex}.

\textbf{Step 1}: By $\phi_i^k(x_i^k)=F_i(x_i^k)$ and \eqref{eq:phikFk}, we obtain
\begin{equation}\label{eq:fplusrhoBexpansion}
	\begin{split}
		&F(\mb{x}^{k+1})-F(\mb{x}^k)\\
		\le&\sum_{i\in\mc{V}}\eta_i\sum_{j\in\bar{\mc{N}}_i} (\phi_j^k(x_j^k+p_{ij}^k)-\phi_j^k(x_j^k)).
	\end{split}
\end{equation}
By the $L$-smoothness of $f_j^k$ and \cite[equation (2.1.7)]{Nesterov04},
\begin{equation*}
	\begin{split}
		&f_j^k(x_j^k+p_{ij}^k)\!-\!f_j^k(x_j^k)\le\langle\nabla f_j^k(x_j^k\!+\!p_{ij}^k), p_{ij}^k\rangle\\
		&-\frac{\|\nabla f_j^k(x_j^k)\!-\!\nabla f_j^k(x_j^k\!+\!p_{ij}^k)\|^2}{2L},
	\end{split}
\end{equation*}
which, together with \eqref{eq:Bijresidualxik}, yields
\begin{equation}\label{eq:BiLipachitzinproof}
	\begin{split}
	    &\phi_j^k(x_j^k\!+\!p_{ij}^k)\!-\!\phi_j^k(x_j^k)\!\le\!\!-\frac{\rho\|\nabla B_j(x_j^k)\!-\!\nabla B_j(x_j^k\!+\!p_{ij}^k)\|^2}{2L_B}\\
		&-\frac{\|\nabla f_j^k(x_j^k)\!-\!\nabla f_j^k(x_j^k\!+\!p_{ij}^k)\|^2}{2L}+\langle\nabla \phi_j^k(x_j^k\!+\!p_{ij}^k),p_{ij}^k\rangle\\
		&\le\! -\frac{\|\nabla \phi_j^k(x_j^k)\!-\!\nabla \phi_j^k(x_j^k\!+\!p_{ij}^k)\|^2}{2(L+\rho L_B)}\!+\!\langle\nabla \phi_j^k(x_j^k\!+\!p_{ij}^k),p_{ij}^k\rangle.
	\end{split}
\end{equation}
In the above equation, the last step is due to the fact: For any $\mb{x},\mb{y}\in \mbb{R}^N$ and $b_1,b_2>0$,
\begin{equation*}
	\frac{\|\mb{x}\|^2}{b_1}+\frac{\|\mb{y}\|^2}{b_2}\ge\frac{\|\mb{x}+\mb{y}\|^2}{b_1+b_2}.
\end{equation*}
Since $\{p_{ij}^k\}_{j\in \bar{\mc{N}}_i}$ is optimal to problem \eqref{eq:pupdate} and $\td{\mc{X}}_j$ $\forall j\in\bar{\mc{N}}_i$ are open, by the KKT condition of problem \eqref{eq:pupdate}, there exists $v_i^k\in\mbb{R}^m$ such that
\begin{equation}\label{eq:localoptcond}
	\nabla \phi_j^k(x_j^k+p_{ij}^k) = A_j^Tv_i^k,~\forall j\in\bar{\mc{N}}_i.
\end{equation}
In addition, problem \eqref{eq:pupdate} requires $\sum_{j\in\bar{\mc{N}}_i} A_jp_{ij}^k = \mb{0}$. Then,
\begin{equation}\label{eq:summationphip}
	\sum_{j\in\bar{\mc{N}}_i}\langle \nabla \phi_j^k(x_j^k+p_{ij}^k), p_{ij}^k\rangle = \langle v_i^k, \sum_{j\in\bar{\mc{N}}_i} A_jp_{ij}^k\rangle = 0.
\end{equation}
By substituting \eqref{eq:summationphip} into \eqref{eq:BiLipachitzinproof}, we obtain
\begin{equation}\label{eq:successivedescent}
	\begin{split}
		&\sum_{j\in\bar{\mathcal{N}}_i}(\phi_j^k(x_j^k+p_{ij}^k)-\phi_j^k(x_j^k))\\
		\le&-\sum_{j\in\bar{\mc{N}}_i}\frac{\|\nabla \phi_j^k(x_j^k)\!-\!\nabla \phi_j^k(x_j^k\!+\!p_{ij}^k)\|^2}{2(L\!+\!\rho L_B))}.
	\end{split}
\end{equation}
Substituting \eqref{eq:successivedescent} into \eqref{eq:fplusrhoBexpansion} gives 
\eqref{eq:Fxk1minusFk}.

\textbf{Step 2:} Fix $i\in\mc{V}$, let $\mc{S}_i^{\perp}$ be the orthogonal complement of $\mc{S}_i$, and define $\mb{z}^k=[(z_1^k)^T, \ldots, (z_n^k)^T]^T$ where
\begin{equation}
    z_j^k = \begin{cases}
        \nabla \phi_j^k(x_j^k+p_{ij}^k), & j\in\bar{\mc{N}}_i,\\
        \nabla F_j(x_j^k), & \text{otherwise}.
    \end{cases}
\end{equation}
Because of \eqref{eq:localoptcond}, $\mb{z}^k\in\mc{S}_i^\perp$. In addition, $\nabla \phi_j^k(x_j^k)=\nabla F_j(x_j^k)$. Then,
\begin{equation*}
	\begin{split}
	    &\sum_{j\in\bar{\mc{N}}_i}\|\nabla \phi_j^k(x_j^k)-\nabla \phi_j^k(x_j^k+p_{ij}^k)\|^2\\
		=&\|\nabla F(\mb{x}^k)-\mb{z}^k\|^2\ge \|P_i\nabla F(\mb{x}^k)\|^2 = \|\nabla F(\mb{x}^k)\|_{P_i}^2.
	\end{split}
\end{equation*}
As a result, \eqref{eq:Wweight} holds.

\subsubsection{Convex case}\label{sssec:prooftheoconvex}

The optimal solution to problem \eqref{eq:compactprob} exists because of the compactness of $\mathcal{C}$ and the convexity of $F$.

By $\operatorname{Null}(A)=\operatorname{Range}(W)$ and $W=W^T\succeq \mb{O}$, we have $\operatorname{Null}(A)=\operatorname{Range}(W^{\frac{1}{2}})$. Suppose that $\tilde{\mb{x}}^\star$ is an optimal solution to problem \eqref{eq:compactprob}. Since $A\mb{x}^k=A\tilde{\mb{x}}^\star = c$, we have
\begin{equation}\label{eq:xkxstarrange}
	\mb{x}^k-\tilde{\mb{x}}^\star\in \operatorname{Null}(A)=\operatorname{Range}(W^{\frac{1}{2}}).
\end{equation}
By the convexity of $F(\mb{x})$ and \eqref{eq:xkxstarrange},
\begin{equation*}
	\begin{split}
		&F(\mb{x}^k)-F(\tilde{\mb{x}}^\star)\le\langle \nabla F(\mb{x}^k), \mb{x}^k-\tilde{\mb{x}}^\star\rangle\\
		=& \langle W^{\frac{1}{2}}\nabla F(\mb{x}^k), (W^{\frac{1}{2}})^\dag(\mb{x}^k\!-\!\tilde{\mb{x}}^\star)\rangle\!\le\!\frac{R}{\sqrt{\lambda_W}}\|\nabla F(\mb{x}^k)\|_W.
	\end{split}
\end{equation*}
By the above equation and \eqref{eq:convexsuccessivedescent} and using \cite[Theorem 2.1.14]{Nesterov04}, we obtain \eqref{eq:theoconvex}.

\subsubsection{Strongly convex}\label{sssec:prooftheostronlyconvex}

Since all the $f_i$'s are $\sigma-$strongly convex, so is $F$ and therefore,
\begin{equation}\label{eq:sc}
	\begin{split}
		&F(\tilde{\mb{x}}^\star)-F(\mb{x}^k)\\
		\ge&\langle \nabla F(\mb{x}^k), \tilde{\mb{x}}^\star-\mb{x}^k\rangle+\frac{\sigma}{2}\|\tilde{\mb{x}}^\star-\mb{x}^k\|^2.
	\end{split}
\end{equation}
Since $\mb{x}^k-\tilde{\mb{x}}^\star\in \operatorname{Range}(W)$ by \eqref{eq:xkxstarrange},
\begin{equation}\label{eq:WpseudoW}
	\begin{split}
		\langle\nabla F(\mb{x}^k), \mb{x}^k\!-\!\tilde{\mb{x}}^\star\rangle\!\!=&\langle \nabla F(\mb{x}^k), (WW^\dag)(\mb{x}^k\!-\!\tilde{\mb{x}}^\star)\rangle\\
		=&\langle (WW^\dag)\nabla F(\mb{x}^k), \mb{x}^k\!-\!\tilde{\mb{x}}^\star\rangle.
	\end{split}
\end{equation}
Let $\mb{a}=(WW^\dag)\nabla F(\mb{x}^k)$ and $\mb{b}=\mb{x}^k-\tilde{\mb{x}}^\star$. We have
\begin{equation*}
	\begin{split}
		\langle \mb{a},\mb{b}\rangle+\frac{\sigma}{2}\|\mb{b}\|^2 &= \frac{1}{2}\left\|\frac{\mb{a}}{\sqrt{\sigma}}+\sqrt{\sigma}\mb{b}\right\|^2-\frac{\|\mb{a}\|^2}{2\sigma}\ge -\frac{\|\mb{a}\|^2}{2\sigma},
	\end{split}
\end{equation*}
which, together with \eqref{eq:sc} and \eqref{eq:WpseudoW}, yields
\begin{equation*}
	F(\mb{x}^k)\!-\!F(\tilde{\mb{x}}^\star)\!\le\!\frac{\|(WW^\dag)\nabla F(\mb{x}^k)\|^2}{2\sigma}\!\le\!\frac{\|\nabla F(\mb{x}^k)\|_W^2}{2\sigma\lambda_W}.
\end{equation*}
Combining the above equation with \eqref{eq:convexsuccessivedescent} gives
\begin{equation*}
	F(\mb{x}^{k+1})\!-\!F(\tilde{\mb{x}}^\star)\!\le\! (1\!-\!\frac{\sigma\lambda_W}{L\!+\!\rho L_B})(F(\mb{x}^k)\!-\!F(\tilde{\mb{x}}^\star)),
\end{equation*}
which further yields the result.

\subsection{Proof of Lemma \ref{lemma:optimalitygapconvex}}\label{ssec:proofofoptimalitygap}
Let $\mathbf{x}^\star$ and $\tilde{\mathbf{x}}^\star$ be the optimal solution of problems \eqref{eq:prob} and \eqref{eq:compactprob}, respectively. We first assume $f(\mathbf{x}')-\underline{f}\le \frac{\epsilon}{2}$ and $\rho\le \frac{\epsilon}{2B(\mathbf{x}')}$. Since $f^\star\ge \underline{f}$ and $f(\tilde{\mathbf{x}}^\star)\le F(\tilde{\mathbf{x}}^\star)\le F(\mathbf{x}')=f(\mathbf{x}')+\rho B(\mathbf{x}')$,
\begin{equation*}
    \begin{split}
        f(\tilde{\mathbf{x}}^\star)-f^\star\le f(\mathbf{x}')-\underline{f}+\rho B(\mathbf{x}')\le \epsilon.
    \end{split}
\end{equation*}

Next, we consider the case where $f(\mathbf{x}')-\underline{f}>\frac{\epsilon}{2}$ and $\rho\le \frac{\epsilon^2}{4(f(\mb{x}')-\underline{f})B(\mathbf{x}')}$. Let $\alpha = \frac{\epsilon}{2(f(\mb{x}')-\underline{f})}\in (0,1)$ and $\mb{x}^{\alpha}=\alpha\mb{x}'+(1-\alpha)\mb{x}^\star$. Clearly, $\mb{x}^\alpha$ is feasible to \eqref{eq:compactprob}. Hence, 
\begin{equation}\label{eq:gap}
    f(\tilde{\mathbf{x}}^\star)\le F(\tilde{\mathbf{x}}^\star)\le f(\mathbf{x}^\alpha)+\rho B(\mathbf{x}^\alpha).
\end{equation}
By the convexity of $f$,
\begin{equation}\label{eq:convexf}
	f(\mb{x}^\alpha)-f^\star \le \alpha(f(\mb{x}')-f^\star)\le \epsilon/2.
\end{equation}
In addition, by the definition of $B$,
\begin{equation}\label{eq:graBinverse}
	B(\mb{x}^\alpha)=\sum_{i\in\mc{V}}\sum_{j=1}^{q_i}\frac{1}{-g_i^j(x_i^\alpha)}.
\end{equation}
Since $\mb{x}^\star$ and $\mb{x}'$ are feasible to problems \eqref{eq:prob} and \eqref{eq:compactprob}, respectively, we have $g_i^j(x_i^\star)\le 0$ and $g_i^j(x_i')<0$, which, together with the convexity of $g_i^j$, gives
\begin{equation}\label{eq:gijupperbound}
	g_i^j(x_i^\alpha)\le \alpha g_i^j(x_i')+(1-\alpha)g_i^j(x_i^\star)\le \alpha g_i^j(x_i')<0.
\end{equation}
Substituting \eqref{eq:gijupperbound} into \eqref{eq:graBinverse} gives
\begin{equation*}
	\begin{split}
		B(\mb{x}^\alpha)\le \sum_{i\in\mc{V}}\sum_{j=1}^{q_i}\frac{1}{-\alpha g_i^j(x_i')} = \frac{B(\mb{x}')}{\alpha} = \frac{2(f(\mb{x}')-\underline{f})B(\mathbf{x}')}{\epsilon}.
	\end{split}
\end{equation*}
Hence, $\rho B(\mathbf{x}^\alpha)\le \epsilon/2$, which, together with \eqref{eq:convexf} and \eqref{eq:gap}, ensures $f(\tilde{\mb{x}}^\star)-f^\star\le \epsilon$.

\bibliographystyle{plain} 
\bibliography{reference}

\end{document}